%% file: Infinite_draughts.tex
\documentclass{amsart}

\title[Transfinite game values in infinite draughts]{Transfinite game values in infinite draughts}

\author{Joel David Hamkins}
\address[Joel David Hamkins]
{Professor of Logic, University of Oxford \&\ Sir Peter Strawson Fellow, University College, Oxford}
         \email{joeldavid.hamkins@philosophy.ox.ac.uk}
         \urladdr{http://jdh.hamkins.org}

\author{Davide Leonessi}
\address[Davide Leonessi]
{St Hugh's College, Oxford}
\email{leonessi@maths.ox.ac.uk}
\urladdr{http://leonessi.org}

\thanks{This article is adapted from chapter 3 of the second author's MSc dissertation~\cite{Leonessi2021:MSc-dissertation-transfinite-game-values-in-infinite-games}, for which he earned a distinction at the University of Oxford in September 2021. Commentary can be made about this article on the first author's blog at \href{http://jdh.hamkins.org/transfinite-game-values-in-infinite-draughts}{http://jdh.hamkins.org/transfinite-game-values-in-infinite-draughts}.}

\usepackage{latexsym,amsfonts,amsmath,amssymb,mathrsfs}
\usepackage{bbm}
\usepackage[hidelinks]{hyperref}
\usepackage{relsize}
\usepackage[rgb,dvipsnames,svgnames]{xcolor}
\usepackage{tikz}
\usetikzlibrary{arrows,arrows.meta,petri,topaths,positioning,shapes,shapes.misc,patterns,calc,decorations.pathreplacing,hobby}
\pgfdeclarelayer{foreground}
\pgfdeclarelayer{background}
\pgfdeclarelayer{boardarrows}
\pgfdeclarelayer{boardgrid}
\pgfdeclarelayer{boardshades}
\pgfsetlayers{boardshades,boardgrid,boardarrows,background,main,foreground}
\usepackage{wrapfig} 
\usepackage{caption} 
\DeclareCaptionStyle{compact}{font=footnotesize,format=plain,name={Fig},labelsep=period,skip=1.5ex,margin=0pt,oneside,justification=centering}
\DeclareCaptionStyle{leftside}{style=compact,margin={0pt,9pt}}
\DeclareCaptionStyle{rightside}{style=compact,margin={9pt,0pt}}
\usepackage{float}
\usepackage[utf8]{inputenc}
\RequirePackage{doi}

\usepackage{enumitem}

\include{MathMacrosJDH}
\usepackage[backend=bibtex,style=alphabetic,maxbibnames=15,maxcitenames=6,dateabbrev=false]{biblatex}

\renewcommand{\UrlFont}{} 
\renewbibmacro{in:}{\ifentrytype{article}{}{\printtext{\bibstring{in}\intitlepunct}}} 
\DeclareFieldFormat{url}{{\UrlFont\url{#1}}} 
\DeclareFieldFormat{urldate}{
  (version \thefield{urlday}\addspace%
  \mkbibmonth{\thefield{urlmonth}}\addspace%
  \thefield{urlyear}\isdot)}
\DeclareFieldFormat{eprint:arxiv}{
\ifhyperref
    {\href{http://arxiv.org/abs/#1}{%
    arXiv\addcolon\nolinkurl{#1}}\iffieldundef{eprintclass}{}{\UrlFont{\mkbibbrackets{\thefield{eprintclass}}}}}
    {arXiv\addcolon\nolinkurl{#1}\iffieldundef{eprintclass}{}{\UrlFont{\mkbibbrackets{\thefield{eprintclass}}}}}}
\bibliography{HamkinsBiblio,MathBiblio,PhilBiblio,WebPosts}
\tikzset{
 Checker/.style={circle,draw,inner sep=2.3mm,transform shape},
 RR/.style={Checker,inner sep=2mm,transform shape,fill=Tomato,draw=RawSienna,thick,double=Tomato!50},  
 rr/.style={Checker,fill=Tomato,thick,draw=RawSienna},     
 BB/.style={Checker,inner sep=2mm,transform shape,fill=CadetBlue!65!black,draw=black,thick,double=CadetBlue!70},  
 bb/.style={Checker,fill=CadetBlue!65!black,thick,draw=black},  
 thick ray/.style={{-{[scale=1]>},shorten >=-1pt,line cap=round,line width=.5mm}},
 >=Stealth,
 Square/.style={text width=1cm,text height=1cm,inner sep=0pt,outer sep=0pt,very thick,line join=bevel,transform shape},
 }

\begin{document}

\begin{abstract}
Infinite draughts, or checkers, is played just like the finite game, but on an infinite checkerboard extending without bound in all four directions. We prove that every countable ordinal arises as the game value of a position in infinite draughts. Thus, there are positions from which Red has a winning strategy enabling her to win always in finitely many moves, but the length of play can be completely controlled by Black in a manner as though counting down from a given countable ordinal.
\end{abstract}
\maketitle
\section{Introducing infinite draughts}

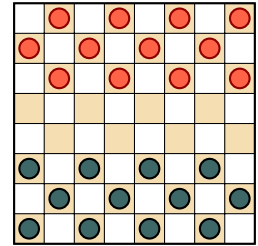
\begin{wrapfigure}[10]{r}{.27\textwidth}\vskip-2ex\hfill
\begin{tikzpicture}[scale=.4] 
 \def\n{8} 
 \def\m{8} 
 \begin{scope}[shift={(-.5,-.5)}] 
  \begin{pgfonlayer}{boardshades}
   \foreach \i in {1,...,\n} {
    \foreach \j in {1,...,\m} {
      \pgfmathsetmacro\c{100*mod(\i+\j,2)};
      \fill [white!\c!Wheat] (\i,\j) rectangle (\i+1,\j+1);
      }
     }
   \draw (1,1) grid (\n+1,\m+1);
   \draw[thick] (1,1) rectangle (\n+1,\m+1);
  \end{pgfonlayer}
 \end{scope}
 \foreach \p/\q in {1/1,3/1,5/1,7/1,2/2,4/2,6/2,8/2,1/3,3/3,5/3,7/3} {\draw (\p,\q) node[bb] {};}
 \foreach \p/\q in {2/6,4/6,6/6,8/6,1/7,3/7,5/7,7/7,2/8,4/8,6/8,8/8} {\draw (\p,\q) node[rr] {};}
\end{tikzpicture}
\captionsetup{style=rightside}
\caption{Finite draughts}
\label{Figure.Finite-draughts}
\end{wrapfigure}
The reader is likely familiar with the game of \emph{draughts}, also commonly known as \emph{checkers}, with players Black and Red taking turns to move their pieces across the checkerboard. Proceeding from the standard starting configuration shown here, both players aim to jump over their opponent's pieces and thereby to capture them---a player wins by capturing all enemy pieces or otherwise placing their opponent into a situation with no available legal moves. Although often played by children, the game of draughts nevertheless admits serious advanced play, and draughts grandmasters compete in international tournaments just as in chess.\footnote{We are grateful to Sergio Scarpetta, current world champion of English Draughts (3-move variant), who kindly reviewed the positions appearing in an earlier version of this paper, confirming our draughts analysis of them.}

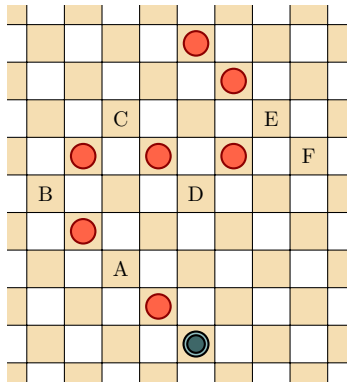
\begin{wrapfigure}[15]{l}{.38\textwidth}\vskip-1.5ex
\begin{tikzpicture}[scale=.5] 
 \def\n{10} 
 \def\m{9} 
 \begin{scope}[shift={(-.5,-.5)}] 
  \begin{pgfonlayer}{boardshades} 
   \clip (1.5,-.5) rectangle (\n+.5,\m+.5);
   \foreach \i in {-1,...,\n} {
    \foreach \j in {-1,...,\m} {
      \pgfmathsetmacro\c{100*mod(\i+\j,2)};
      \fill [white!\c!Wheat] (\i,\j) rectangle (\i+1,\j+1);
      }
     }
   \draw (-1,-1) grid (\n+1,\m+1);
  \end{pgfonlayer}
 \end{scope}
 \foreach \p/\q in {6/0} {\draw (\p,\q) node[BB] {};}
 \foreach \p/\q in {5/1,3/3,3/5,5/5,7/5,7/7,6/8} {\draw (\p,\q) node[rr] {};}
 \draw[every node/.style={scale=.8}]
     (4,2) node {A}
     (2,4) node {B}
     (4,6) node {C}
     (6,4) node {D}
     (8,6) node {E}
     (9,5) node {F};
\end{tikzpicture}\hfill
\captionsetup{style=leftside}
\caption{A position in infinite draughts, Black to play}
\label{Figure.A-position-in-infinite-draughts}
\end{wrapfigure}
\enlargethispage{40pt}
In this article we focus on \emph{infinite draughts}, an infinitary version of the game played on an infinite checkerboard, extending without bound in all four directions, like the integer lattice. Play proceeds indefinitely, perhaps infinitely. Although we might imagine a friendly game of infinite draughts at the cafe, the game is more often considered by curious mathematicians, who investigate what it would be like to proceed from one position or another. Infinite draughts, we shall prove, exhibits a robust transfinite game value phenomenon---our main result is that every countable ordinal arises as the game value of a position in infinite draughts, and this result is optimal for games having countably many options at each move. In short, the omega one of infinite draughts is true omega one.\goodbreak

Let us begin by getting clear on the rules of the game and the nature of infinite play. Infinite draughts has no standard starting configuration. Instead, we analyze how the game might proceed from any given board position. In the position of figure \ref{Figure.A-position-in-infinite-draughts}, for example, the black king can undertake iterated captures via ABCDE, but after this Red will recapture by jumping to~F for the win.

To be precise, a player captures an enemy piece by jumping over it, if the space just beyond is empty; captured pieces are  removed immediately from the board. Jumping can occur iteratively, with a single piece performing  multiple jumps in succession, such as with the move ABCDE we just discussed, and iterated jumps count altogether as a single turn.

There are two types of pieces, pawns and kings, with the kings indicated in our diagrams with a double circular boundary. The difference is that pawns are allowed forward movement only, with black pawns proceeding on upward diagonals and red pawns downward, while kings can move freely upward or downward. All play takes place on the dark squares only. In finite draughts, a pawn is promoted to king upon reaching the opposite side of the board, but in infinite draughts, regrettably, since there are no boundary edges to the infinite board, such promotions do not occur. Nevertheless, we shall freely consider positions in infinite draughts involving both pawns and kings. The full ABCDE iterated jumping move in the previous figure, for example, would have been impossible if the black piece had been merely a pawn, since the jump from C to D was regressive.

We shall play with the \emph{forced-jump} rule, which means that if a player can jump an enemy piece, then it is obligatory to make such a jumping move. If there is a choice of jumping moves, however, the player may freely choose which jump to make. We play furthermore with the \emph{forced-iterated-jump} rule, which means that a player must continue to make jumps in iteration, when this is possible, performing a maximal such jump. That is, if after making a jumping move or an iteration of jumping moves, yet another jump remains possible with that same piece, then it is obligatory to make such an additional jump, until a situation is realized in which no additional jumps are possible with that piece. Such an iterated jump sequence counts as a single turn for the player undertaking it.

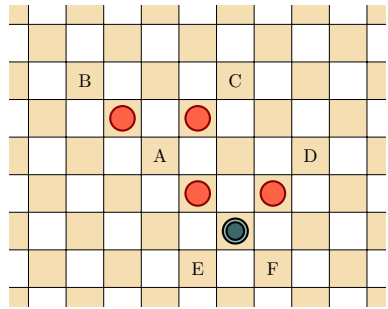
\begin{wrapfigure}{r}{.43\textwidth}\hfill
\begin{tikzpicture}[scale=.5] 
 \def\n{10} 
 \def\m{8} 
 \clip (0,0) rectangle (\n,\m);
 \begin{scope}[shift={(-.5,-.5)}] 
  \begin{pgfonlayer}{boardshades} 
   \clip (.5,.5) rectangle (\n+.5,\m+.5);
   \foreach \i in {0,...,\n} {
    \foreach \j in {0,...,\m} {
      \pgfmathsetmacro\c{100*mod(\i+\j,2)};
      \fill [white!\c!Wheat] (\i,\j) rectangle (\i+1,\j+1);
      }
     }
   \draw[very thin] (0,0) grid (\n+1,\m+1);
  \end{pgfonlayer}
 \end{scope}
 \draw (6,2) node[BB] {};
 \foreach \p/\q in {5/3,7/3,3/5,5/5}
   {\draw (\p,\q) node[rr] {};}  
 \draw[every node/.style={scale=.7}]
     (4,4) node {A}
     (2,6) node {B}
     (6,6) node {C}
     (8,4) node {D}
     (5,1) node {E}
     (7,1) node {F};
 \begin{pgfonlayer}{boardarrows}
   \begin{scope}[Blue,->,shorten >=-4pt]
   \end{scope}
 \end{pgfonlayer}
\end{tikzpicture}
\captionsetup{style=rightside}
\caption{A choice of jump iterations}
\label{Figure.A-choice-of-interations}
\end{wrapfigure}
When a branching choice of jumps for a piece is possible, a player may legally follow any of the branching paths, even if doing so would ultimately result in a smaller total number of jumps. In the position here, for example, with both the forced-jump and forced-iterated-jump rules, then Black has three options: the iterated jump AB, the iterated jump AC, or the single jump to D. With only the forced jump rule but not the forced-iterated-jump rule, then Black also has the option of the single jump to A. With no forced-jump rules at all, however, Black would also have the two additional options of a simple move to E or to F.

\newpage
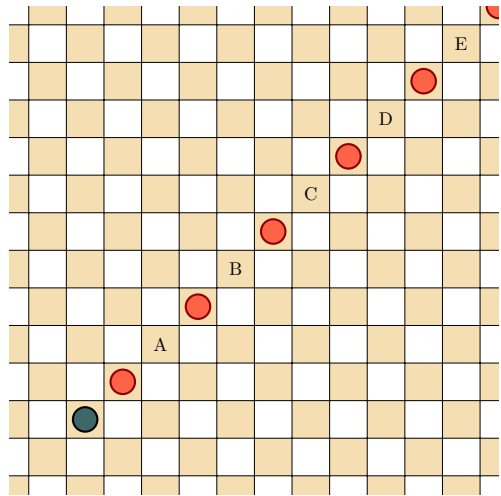
\begin{wrapfigure}{r}{.53\textwidth}\hfill
\begin{tikzpicture}[scale=.5] 
 \def\n{13} 
 \def\m{13} 
 \clip (0,0) rectangle (\n,\m);
 \begin{scope}[shift={(-.5,-.5)}] 
  \begin{pgfonlayer}{boardshades} 
   \clip (.5,.5) rectangle (\n+.5,\m+.5);
   \foreach \i in {0,...,\n} {
    \foreach \j in {0,...,\m} {
      \pgfmathsetmacro\c{100*mod(\i+\j,2)};
      \fill [white!\c!Wheat] (\i,\j) rectangle (\i+1,\j+1);
      }
     }
   \draw[very thin] (0,0) grid (\n+1,\m+1);
  \end{pgfonlayer}
 \end{scope}
 \draw (2,2) node[bb] {};
 \foreach \p/\q in {3/3,5/5,7/7,9/9,11/11,13/13}
   {\draw (\p,\q) node[rr] {};}  
 \draw[every node/.style={scale=.7}]
     (4,4) node {A}
     (6,6) node {B}
     (8,8) node {C}
     (10,10) node {D}
     (12,12) node {E};
 \begin{pgfonlayer}{boardarrows}
   \begin{scope}[Blue,->,shorten >=-4pt]
   \end{scope}
 \end{pgfonlayer}
\end{tikzpicture}
\captionsetup{style=rightside}
\caption{An infinite iterated jumping opportunity}
\label{Figure.Infinite-iterated-jump}
\end{wrapfigure}
The possibility of iterated jumps in infinite draughts leads to a surprising possibility and conundrum not arising in the finite game. Namely, what about infinite iterated jumps? In the position here, for example, the black pawn can begin capturing the red pawns on the main diagonal, jumping once to A, then to B, and so on via CDE. If we suppose the pattern continues forever up the diagonal, then it seems that Black could undertake an infinitely iterated jump without end. Every single red piece on the diagonal will be thereby captured on this move. Fine---all the red pieces will be captured. The conundrum concerns the black piece, namely, where is it at the end of the move? The \emph{infinite-jump} rule of infinite draughts answers this question by stating that when a piece undertakes an infinite iterated jump, then the active piece itself also is removed from the board after the jumps take place. It is as though the black piece jumps off to infinity, thereby not only capturing all those red pieces, but also removing itself from the board. While jumping, the black piece will have long abandoned any particular square, and so it is not to be found on any of them.

This rule has the intriguing consequence that a player can lose a piece not only by being captured, during their opponent's turn, but also by jumping to infinity on their own turn, a form of self-capture. In particular, for the position here, after the black iterated jumping move, the board will be completely empty of pieces. Since it would be Red's turn, this would mean a loss for Red. But if there were another red piece off the diagonal, then it would be a loss for Black, since Red would be able to move and then it would be Black's turn, but he has no pieces.

Let us clarify the winning condition more exactly. A player loses a game of infinite draughts when it is their turn, but they have no legal move available. This situation might arise simply because all their pieces have been captured and they have no piece left to move; but it can also happen when they do have pieces, if those pieces should happen to be all boxed in by the position with no legal move available. When a player has lost, the other player wins. Meanwhile, infinitely long play always counts as a draw, and so in infinite draughts there is no need for the 40-move or 50-move rules sometimes adopted in tournaments---those rules, in our view, should be seen ultimately as a proxy for draw by infinite play. Notice that every win of infinite draughts occurs at a finite stage of play, making it what is called an open game, which we shall elaborate upon in the next section.

\section{A brief review of transfinite game values}

Let us briefly review the theory of transfinite ordinal game values. An introduction to game values in the context of infinite chess is provided in~\cite{EvansHamkins2014:TransfiniteGameValuesInInfiniteChess, EvansHamkinsPerlmutter:APositionInInfiniteChessWithGameValueOmega^4}, but the idea is applicable to many games, and our main aim in this article is to exhibit high game values in infinite draughts. The theory of ordinal game values can be used to prove the Gale-Stewart theorem on open determinacy~\cite{GaleStewart1953:InfiniteGamesWithPerfectInformation}, and game values are also closely related to the theory of Sprague-Grundy nimbers.

\begin{wrapfigure}[12]{r}{.35\textwidth}\vskip-1.5ex\hfill
\begin{tikzpicture}[scale=.5] 
 \def\n{8} 
 \def\m{8} 
 \clip (0,0) rectangle (\n,\m);
 \begin{scope}[shift={(-.5,-.5)}] 
  \begin{pgfonlayer}{boardshades} 
   \clip (.5,.5) rectangle (\n+.5,\m+.5);
   \foreach \i in {0,...,\n} {
    \foreach \j in {0,...,\m} {
      \pgfmathsetmacro\c{100*mod(\i+\j,2)};
      \fill [white!\c!Wheat] (\i,\j) rectangle (\i+1,\j+1);
      }
     }
   \draw[very thin] (0,0) grid (\n+1,\m+1);
  \end{pgfonlayer}
 \end{scope}
 \draw (6,2) node[bb] {};
 \foreach \p/\q in {4/4,3/5,2/6}
   {\draw (\p,\q) node[rr] {};}
 \draw[every node/.style={scale=.7}]
      (5,3) node {A};
\end{tikzpicture}
\captionsetup{style=rightside}
\caption{Game value 2}
\label{Figure.Game-value-2}
\end{wrapfigure}
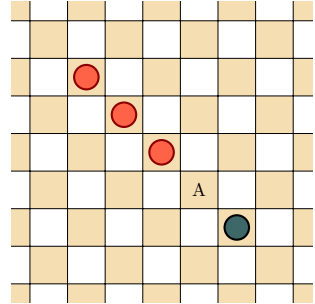
Game values generalize to the transfinite the concept of a mate-in-$n$ position in chess, occurring when a player can force a winning checkmate in $n$ moves but not fewer---such a position has game value $n$. To illustrate the idea in draughts, consider the simple position shown here in figure \ref{Figure.Game-value-2}, with Red to play. The game value is $2$ for Red, since Red can advance the leading red pawn down to A as a kind of bait, obligating Black under the forced-jump rule to capture it, after which Red is in a position immediately to recapture and thereby win. So Red can force a win in two moves, and not fewer, and this is precisely what it means to have a game value of~$2$ for that player.

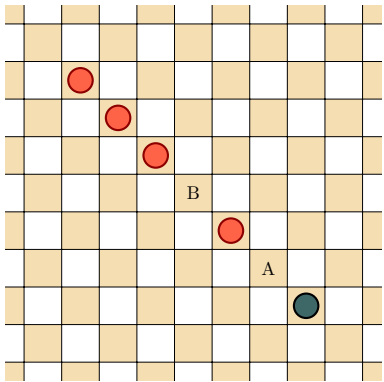
\begin{wrapfigure}[15]{l}{.42\textwidth}
\begin{tikzpicture}[scale=.5] 
 \def\n{10} 
 \def\m{10} 
 \clip (0,0) rectangle (\n,\m);
 \begin{scope}[shift={(-.5,-.5)}] 
  \begin{pgfonlayer}{boardshades} 
   \clip (.5,.5) rectangle (\n+.5,\m+.5);
   \foreach \i in {0,...,\n} {
    \foreach \j in {0,...,\m} {
      \pgfmathsetmacro\c{100*mod(\i+\j,2)};
      \fill [white!\c!Wheat] (\i,\j) rectangle (\i+1,\j+1);
      }
     }
   \draw[very thin] (0,0) grid (\n+1,\m+1);
  \end{pgfonlayer}
 \end{scope}
 \draw (8,2) node[bb] {};
 \foreach \p/\q in {6/4,4/6,3/7,2/8}
   {\draw (\p,\q) node[rr] {};}
 \draw[every node/.style={scale=.7}]
      (7,3) node {A}
      (5,5) node {B};
\end{tikzpicture}\hfill
\captionsetup{style=leftside}
\caption{Game value 3}
\label{Figure.Game-value-3}
\end{wrapfigure}
The position in figure \ref{Figure.Game-value-3}, in contrast, has game value $3$ for Red, with Red to play. Namely, Red can advance the isolated pawn to A, obligating Black to jump, and this in effect reduces to the previous position, since afterward Red can advance a pawn to B as more bait, obligating Black again to jump, after which Red recaptures for the win. Red can therefore force a win in $3$ moves, and not fewer, and so this position has game value $3$ for Red. We invite the reader to carry the pattern further and create positions in infinite draughts exhibiting game value $n$ for any desired finite number $n$.

Things become truly fascinating with infinite game values. The position shown in figure~\ref{Figure.Value-omega}, we claim, illustrates a game value of $\omega$. Let us explain.
\begin{figure}\centering
\begin{tikzpicture}[scale=.45] 
 \def\n{30} 
 \def\m{44} 
 \clip (2,4) rectangle (30,44);
 \useasboundingbox (2,4) rectangle (30,44);
 \begin{scope}[shift={(-.5,-.5)}] 
  \begin{pgfonlayer}{boardshades} 
   \clip (2.5,4.5) rectangle (\n+.5,\m+.5);
   \foreach \i in {0,...,\n} {
    \foreach \j in {0,...,\m} {
      \pgfmathsetmacro\c{100*mod(\i+\j,2)};
      \fill [white!\c!Wheat] (\i,\j) rectangle (\i+1,\j+1);
      }
     }
  \end{pgfonlayer}
  \begin{pgfonlayer}{boardgrid} 
    \clip (2.5,4.5) rectangle (\n+.5,\m+.5);
    \draw[thin] (0,0) grid (\n+1,\m+1);
   \end{pgfonlayer}
 \end{scope}
 \draw (6,6) node[BB] {};
 \foreach \p/\q in {
  7/7,9/9,11/11,13/13,15/15,17/17,19/19,21/21,23/23,25/25,27/27,29/29,
  7/9,5/11,4/12,
  11/13,9/15,6/18,5/19,4/20,
   17/19,15/21,13/23,10/26,8/28,7/29,6/30,
   25/27,23/29,21/31,19/33,16/36,14/38,12/40,11/41,10/42,
   }
   {\draw (\p,\q) node[rr] {};}  
 \draw (6,10) node[scale=.8] (1) {$1$};
 \draw (8,16) node[scale=.8] (2) {$2$};
 \draw (12,24) node[scale=.8] (3) {$3$};
 \draw (18,34) node[scale=.8] (4) {$4$};
 \begin{scope}[transparency group,opacity=.5]
 \end{scope}
 \begin{pgfonlayer}{boardshades}
  \clip (0,0) rectangle (\n,\m);
   \foreach \p/\q in {8/8,10/10,12/12,14/14,16/16,18/18,20/20,22/22,24/24,26/26,28/28,30/30
   }
     {\draw (\p,\q) node[Square,fill=CadetBlue!60!Blue!40] {}; }
   \foreach \p/\q in {
   6/10,
   10/14,8/16,
   16/20,14/22,12/24,
   24/28,22/30,20/32,18/34
   }
     {\draw (\p,\q) node[Square,fill=Orchid!80!Blue!40] {}; }
 \end{pgfonlayer}
 \begin{pgfonlayer}{boardarrows}
  \draw[CadetBlue!60!Navy,very thick,->] (6,6) -- (\n+2,\n+2);
  \begin{scope}[Orchid!50!Navy,very thick,->,shorten >=25pt]
     \draw (8,8) -- (1.center);
     \draw (12,12) -- (2);
     \draw (18,18) -- (3);
     \draw (26,26) -- (4);
  \end{scope}
 \end{pgfonlayer}
\end{tikzpicture}
\caption{A position in infinite draughts with game value $\omega$}\label{Figure.Value-omega}
\end{figure}
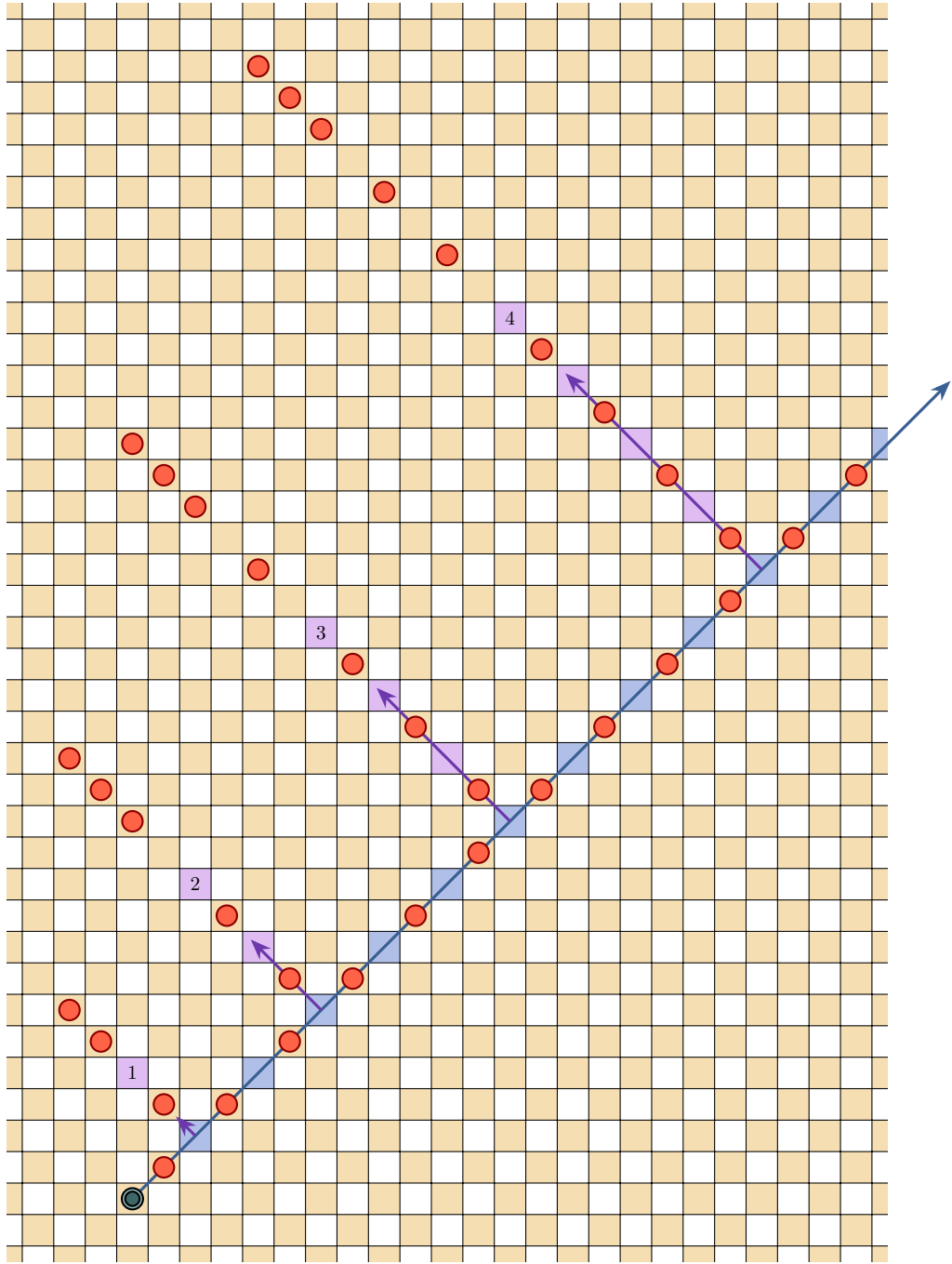
It is Black to play, and since a jumping move is available, it is therefore obligatory for Black to jump. Indeed, a long iteration of jumping moves is possible by proceeding straight up the blue-gray highlighted diagonal, what we shall call the main ladder. We assume that the pattern of the position continues indefinitely, and so this is an infinite iterated jumping opportunity. But if Black should undertake all those infinitely iterated jumps up the main ladder, however, then Black will have jumped away to infinity, and because of the infinite-jump rule, therefore, the black piece would disappear, while leaving some red pieces surviving, which would mean a loss for Black on his next move.

Black can avoid this immediate loss by opting instead to exit the main ladder onto one of the various branching offshoots, indicated in violet, coming to rest either at square $1$, or $2$, or $3$, and so on. We intend that the pattern of the position continues indefinitely, and that Black can opt to turn off at any point and come to rest at square $n$, for any desired finite number $n$.

If Black should come to rest at square $1$, then the game value of the resulting position would be exactly $1$, for Red could immediately recapture, causing Red to win after exactly one additional move. Alternatively, if Black should come to rest at square $2$, then locally this position would be just the same as the position we had considered at the outset of this section, with game value $2$. Namely, Red can advance down, obligating a forced-jump reply, after which Red  recaptures for the win. So if Black opts to settle upon square $2$, then Red forces a win in exactly two additional moves, making the resulting game value $2$. Similarly, if Black should instead come to rest instead at square position $3$, then Red can force a win in three additional moves, giving that position game value $3$. And more generally, for every number $n$, if Black should opt on the first move to come to rest at square $n$, then Red will have a sequence of $n$ additional moves, each with a forced-jump reply, leading to an inevitable Red win. Notice that the finite columns branching off the main ladder are spaced increasingly far apart, and so there is no possibility that Red could gain a quicker win than by following the line of play we have described.

We may summarize the interesting situation of this position as follows. It is Black's turn, and he must either undertake the infinite iterated jump, causing an immediate loss, or else settle upon square $n$ for some finite $n$, after which Red can win in exactly $n$ moves, and not fewer. Thus, Red has a winning strategy that will always win in finitely many moves, but Black controls how long it will take. Namely, Red will win, but by the choice of his first move, Black can postpone the inevitable defeat by any desired finite amount. This combination of features is precisely what it means for a game to exhibit game value $\omega$, the first infinite ordinal.

With these examples now under our belt, let us give the precise general definition of transfinite ordinal game value. First, we define that an infinite two-player game of perfect information is \emph{open} for a player, if all possible ways for that player to win are known already as definite wins at a finite stage of play. (The terminology arises from the fact that this condition is equivalent to the set of winning plays for that player forming an open set in the product topology in the space of all infinite plays.) Infinite chess, for example, is an open game for both players, because checkmate, when it occurs, does so at a finite stage of play. Infinite draughts also is an open game, because one wins the game only by placing the opponent into a situation with no legal move on their turn, and this occurs at a finite stage of play if it does at all---all infinitely long play counts as a draw.

In any open game, we now define the ordinal game values of positions in that game. Game values are defined relative to the perspective of one player, whose winning condition is open, and we shall refer to this player as the open player and the other player as the closed player. The game may happen to be open for both players, as in infinite chess and infinite draughts.

\begin{definition}
In any open game, we define the ordinal \emph{game value} of positions in the game by the following transfinite recursion:
\begin{enumerate}
    \item If in a position the game is already won for the open player, then the game value of the position is $0$.
    \item If the game is not yet won and it is the open player's turn, and he can move to a position with value $\alpha$, then for the smallest such ordinal $\alpha$, the value of the position is $\alpha+1$.
    \item If it is the closed player's turn, then the value of the position is the supremum of the values of the positions to which a legal move can be made, if all such positions have a value, otherwise the value is not yet defined.
\end{enumerate}
\end{definition}
Thus, the notion ``position $p$ has game value $\alpha$'' is defined by transfinite ordinal recursion on $\alpha$. Some positions may never happen to have a value assigned to them, and this is fine---these are the unvalued positions for the open player.

The fundamental observation of game values is that if a game position has a value, then the open player has a winning strategy for play from that position, namely, the value-reducing strategy, according to which the open player should play always so as to reduce the game value. If a position has a value, then the open player can indeed reduce the value by one (unless she has already won), and the closed player can never move from a valued position to an unvalued position or to a position with higher value. Thus, the value-reducing strategy leads to a steadily decreasing sequence of ordinal values, which must eventually reach $0$, meaning that the open player has won. So the value-reducing strategy is a winning strategy for the open player from any position having a value.

If in contrast a game position does not have a value, then the closed player has a strategy that will prevent an open-player win, the \emph{value-avoiding} strategy, by which the closed player aims to remain always on unvalued positions. The open player can never move from an unvalued position to a valued position, for if this would be possible then the position would have had a value after all, and if a position is unvalued and it is the closed player's turn, then there must be a move to an unvalued position, for otherwise again the position would have had a value. The value-avoiding strategy therefore ensures that the game play will remain on unvalued positions, and in particular, it ensures that the open player does not win, because for the open player to win requires realizing value $0$ at a finite stage of play. In a game without draws, the value-avoiding strategy is therefore a winning strategy for the closed player from any unvalued game position.

The analysis we have just given in the two preceding paragraphs amounts to a proof of the Gale-Stewart theorem on open determinacy---every open game without draws is determined---since if the initial position has a value, then the open player can win by the value-reducing strategy, and if it doesn't, then the closed player can win by the value-avoiding strategy.

\section{High game values arise in infinite draughts}

We now come to the main contribution of this article, our proof that infinite draughts exhibits the most robust possible spectrum of game values.

\begin{theorem}\label{MainTheorem}
Every countable ordinal arises as the game value of a position in infinite draughts.
\end{theorem}

The main idea of the proof will be to embed certain well-founded trees into positions of infinite draughts, in such a manner that the way play unfolds in the draughts game corresponds to Black climbing in the tree, with Black losing when a terminal node is reached. This was the same strategy employed by the first author in~\cite[theorem~10]{EvansHamkins2014:TransfiniteGameValuesInInfiniteChess} to prove that every countable ordinal arises as the game value of a position in infinite 3D chess. The point is that every countable ordinal arises as the ordinal rank of a well-founded tree $T$ on the natural numbers, that is a tree of finite sequences of natural numbers with no infinite branch, and furthermore the game value of such a climbing-through-$T$ game is exactly the ordinal rank of the tree $T$; it follows that the ordinal game values of such infinite draughts positions are at least as great as the ranks of the corresponding trees, and we can conclude that every countable ordinal arises as the game value of a position in infinite draughts.\goodbreak

\newpage

\begin{wrapfigure}{r}{.6\textwidth}\hfill
\begin{tikzpicture}[scale=.5] 
 \def\n{14} size of board n= number of cols
 \def\m{16} 
 \clip (0,0) rectangle (\n,\m);
 \begin{scope}[shift={(-.5,-.5)}] 
  \begin{pgfonlayer}{boardshades} 
   \clip (.5,.5) rectangle (\n+.5,\m+.5);
   \foreach \i in {0,...,\n} {
    \foreach \j in {0,...,\m} {
      \pgfmathsetmacro\c{100*mod(\i+\j,2)};
      \fill [white!\c!Wheat] (\i,\j) rectangle (\i+1,\j+1);
      }
     }
   \draw[very thin] (0,0) grid (\n+1,\m+1);
  \end{pgfonlayer}
  \begin{pgfonlayer}{boardgrid} 
    \clip (.5,.5) rectangle (\n+.5,\m+.5);
    \draw[thin] (0,0) grid (\n+1,\m+1);
  \end{pgfonlayer}
 \end{scope}
 \draw (7,1) node[BB] {};
 \foreach \p/\q in {6/2,6/4,8/6,10/8,12/10,14/12,16/14,
 8/8,6/10,4/12,2/14,0/16,
 6/12,8/14,10/16
 }
   {\draw (\p,\q) node[rr] {};}
 \begin{pgfonlayer}{boardshades}
  \clip (0,0) rectangle (\n,\m);
   \foreach \p/\q in {5/3,7/5,9/7,11/9,13/11,15/13,
     7/9,5/11,3/13,1/15,
     7/13,9/15
     }
     {\draw (\p,\q) node[Square,fill=CadetBlue!60!Blue!40] {}; }
  \end{pgfonlayer}
  \begin{pgfonlayer}{boardarrows}
    \begin{scope}[CadetBlue!50!Navy,very thick,->,shorten >=-6pt]
     \draw (7,1) -- (5,3) -- (13,11);
     \draw (9,7) -- (1,15);
     \draw (5,11) -- (9,15);
    \end{scope}
  \end{pgfonlayer}
\end{tikzpicture}
\captionsetup{style=rightside}
\caption{The infinite binary branching tree can\\ be represented in infinite draughts}
\label{Figure.Binary-tree}
\end{wrapfigure}
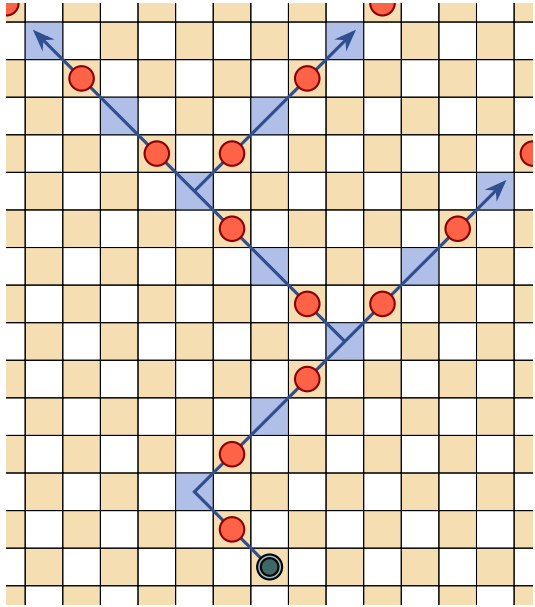
Let us begin by observing that we may easily embed the infinite binary branching tree into draughts, in the sense that there is a infinite draughts position, such as the one shown here, in which a Black king faces a choice of iterated jumping paths having the same branching structure as the full binary branching tree---every node in the tree leads eventually to a further branching node. We may furthermore arrange that the branches of this tree become increasingly separated from one another on the board, so that as you climb the tree there are increasingly long stretches with no other parts of the tree nearby. To achieve this, simply allow the current branches to continue growing apart for a good length of time before branching again, and in this way the branches of the tree become increasingly far apart.  There is truly plenty of room in the infinite plane. Observe also that because we can embed the full binary tree, we can also embed any subtree of it. We can also embed these trees just as easily in a downward-oriented direction.

To achieve the high game values, however, we shall want to embed not binary branching trees, but rather infinitely branching trees, more specifically, well-founded trees on the natural numbers. For this, we use the idea that is present already in figure~\ref{Figure.Value-omega}. Namely, the Black king in that position is in effect faced with an infinitely branching choice: he can choose to climb the main ladder to infinity, thereby dying immediately and losing the game, or he can come to rest at one of the squares $n$ for any desired $n$. In our argument here, we shall arrange that those squares $n$ lead not just to a position with value $n$, but to a further part of the tree with a desired game value $\alpha_n$. By doing so, we shall thereby create a position with game value exceeding all those values $\alpha_n$.

In this way, we shall thereby prove by transfinite induction that every countable ordinal arises as the game value of a position in infinite draughts in the form of a tree-like position with a single Black king, faced with a forced-jump move. We have already exhibited such positions with all finite values, as well as value $\omega$. If there are such positions with value $\alpha_n$ for every $n$, then we can make a position with value $\sup\set{\alpha_n+1\mid n\in\omega}$ by creating a main ladder as in figure~\ref{Figure.Value-omega}, and having square $n$ be a node where the subsequent tree has value $\alpha_n$.\goodbreak

Let us get into the details. It turns out that our main proof strategy is a little easier to implement for the version of infinite draughts without the iterated forced-jump rule. And so before proving theorem \ref{MainTheorem}, let us first prove the following theorem to illustrate the proof method.

\begin{theorem}\label{Theorem.No-forced-iteration}
 Every countable ordinal arises as the game value of a position in the version of infinite draughts with the forced-jump rule, but without the forced-iterated-jump rule.
\end{theorem}

\smallskip\noindent\emph{Proof.} 
To clarify, in this version of the game, when a player is faced with a jumping-move possibility, then it is obligatory to make a jumping move, and one is free to continue with iterated jumping, including infinitely iterated jumping, if possible, but there is no obligation to continue jumping after the initial jump on a turn and one is free to stop at any point along the iterated jump.

\begin{wrapfigure}{r}{.42\textwidth}\vskip-2ex\hfill
\begin{tikzpicture}[scale=.5] 
 \def\n{10} size of board n= number of cols
 \def\m{10} 
 \useasboundingbox (0,-.5) rectangle (\n,\m);
 \clip (0,0) rectangle (\n,\m);
 \begin{scope}[shift={(-.5,-.5)}] 
  \begin{pgfonlayer}{boardshades} 
   \clip (.5,.5) rectangle (\n+.5,\m+.5);
   \foreach \i in {-1,...,\n} {
    \foreach \j in {-1,...,\m} {
      \pgfmathsetmacro\c{100*abs(mod(\i+\j,2))};
      \fill [white!\c!Wheat] (\i,\j) rectangle (\i+1,\j+1);
      }
     }
  \end{pgfonlayer}
  \begin{pgfonlayer}{boardgrid} 
    \clip (.5,.5) rectangle (\n+.5,\m+.5);
    \draw[thin] (0,0) grid (\n+1,\m+1);
  \end{pgfonlayer}
 \end{scope}
 \draw (2,8) node[BB] (BB) {};
 \foreach \p/\q in {3/7,5/5,7/3,9/1,7/7}
   {\draw (\p,\q) node[RR] {};}
 \begin{pgfonlayer}{boardshades}
  \clip (0,0) rectangle (\n,\m);
   \foreach \p/\q in {4/6,6/4,8/2,10/0}
     {\draw (\p,\q) node[Square,fill=CadetBlue!60!Blue!40] {}; }
  \end{pgfonlayer}
  \begin{pgfonlayer}{boardarrows}
    \draw[CadetBlue!50!Navy,very thick,->] (BB) -- ++(8.5,-8.5);
  \end{pgfonlayer}
\end{tikzpicture}
\captionsetup{style=rightside}
\caption{This infinite descending\\ ladder has game value 1}
\label{Figure.Infinite-descending-ladder}
\end{wrapfigure}
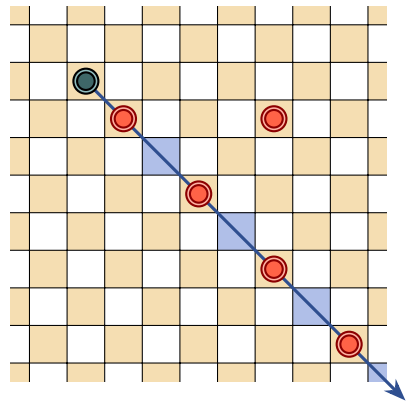
To be sure, it can often be dangerous for a draughts player to halt in the middle of an iterated jump, since one might be faced with immediate recapture. In the position here, for example, Black has an infinite iterated jump opportunity, straight down the blue ladder, but it would be unsafe for him to stop at any point along the way, since he would be facing a red king who could immediately capture him. So this position has game value $1$ for Red, since either Black will jump infinitely and then disappear, or else Black will stop at some point, only to be immediately captured by Red. So Red can ensure that in any case, Black is without a move on his next turn. Nevertheless, in some positions it can be safe to stop in the middle of an iterated jump. If some of those red kings had been pawns, for example, then Black could have safely stopped just above them, since the red pawns move only downward; and if all the red pieces had been pawns, one can see that the position would be a draw without the forced-iterated-jump rule. Without the extra red piece, the position would be a win for Black via the infinite iterated-jump, since this would leave Red without any move, even though Black also would have no pieces.

Let us prove the theorem. Suppose that we have  tree-like positions in which Black sits at the apex of a tree, and optimal play will call for Black to descend on paths of the tree, with game values of $\alpha_1$, $\alpha_2$, $\alpha_3$, and so on for Red, with Black to move first. (We have already illustrated how to achieve game value $1$, and game value $0$ is a position where Red has already won.) Let us assume that this sequence is nondecreasing, so that $\alpha_1\leq\alpha_2\leq\alpha_3$ and so on. We shall now construct a position, depicted in figure \ref{Figure.Infinite-branching-ladder}, with value strictly exceeding these. Black is situated at the top of the main ladder, descending to the right in blue, which has offshoot branching nodes at square $1$, square $2$, and so on, with each square $n$ leading to our tree positions with game value $\alpha_n$. All the red pieces are kings, except for the two pawns below each branching node, making those squares at least momentarily safe for Black to stop an iterated jump. When $\alpha_n>1$, then the colored subtree we have attached at square $n$ will itself have such branching nodes.

\begin{figure}[h]\centering
\begin{tikzpicture}[scale=.5] 
 \def\n{18} size of board n= number of cols
 \def\m{18} 
 \useasboundingbox (-1,-1) rectangle (\n+1,\m);
 \clip (0,0) rectangle (\n,\m);
 \begin{scope}[shift={(-.5,-.5)}] 
  \begin{pgfonlayer}{boardshades} 
   \clip (.5,.5) rectangle (\n+.5,\m+.5);
   \foreach \i in {-1,...,\n} {
    \foreach \j in {-1,...,\m} {
      \pgfmathsetmacro\c{100*abs(mod(\i+\j,2))};
      \fill [white!\c!Wheat] (\i,\j) rectangle (\i+1,\j+1);
      }
     }
  \end{pgfonlayer}
  \begin{pgfonlayer}{boardgrid} 
    \clip (.5,.5) rectangle (\n+.5,\m+.5);
    \draw[thin] (0,0) grid (\n+1,\m+1);
  \end{pgfonlayer}
 \end{scope}
 \draw (2,16) node[BB] (BB) {};
 \foreach \p/\q in {3/15,7/11,11/7,13/5,17/1,
  1/11,
  5/7,3/5,1/3,
  11/1
 }
   {\draw (\p,\q) node[RR] {};}
 \foreach \p/\q in {5/13,9/9,15/3,
  3/13,
  7/9,
    13/3
 }
   {\draw (\p,\q) node[rr] {};}
 \begin{pgfonlayer}{boardshades}
  \clip (0,0) rectangle (\n,\m);
   \foreach \p/\q in {4/14,6/12,8/10,10/8,12/6,14/4,16/2,18/0}
     {\draw (\p,\q) node[Square,fill=CadetBlue!60!Blue!40] {}; }
   \foreach \p/\q in {2/12,0/10}
     {\draw (\p,\q) node[Square,fill=Orchid!80!Blue!40] {}; }
   \foreach \p/\q in {6/8,4/6,2/4,0/2}
     {\draw (\p,\q) node[Square,fill=SpringGreen!70] {}; }
   \foreach \p/\q in {12/2,10/0}
     {\draw (\p,\q) node[Square,fill=Yellow!50!Orange!70] {}; }
  \end{pgfonlayer}
  \begin{pgfonlayer}{boardarrows}
  \draw[CadetBlue!50!Navy] (4,14) node[Square,draw=CadetBlue!50!Navy,very thick] (1) {} node[black,scale=.8] {1}
        (8,10) node[Square,draw,very thick] (2) {} node[black,scale=.8] {2}
        (14,4) node[Square,draw,very thick] (3) {} node[black,scale=.8] {3};
    \draw[CadetBlue!50!Navy,very thick,->] (BB) -- (1) -- (2) -- (3) -- ++(5,-5);
    \draw[Orchid!50!CadetBlue!50!Navy,very thick,->] (1) -- ++(-5,-5) node[below] {$\alpha_1$};
    \draw[DarkGreen,very thick,->] (2) -- ++(-9,-9) node[below] {$\alpha_2$};
    \draw[RawSienna,very thick,->] (3) -- ++(-5,-5) node[left] {$\alpha_3$};
  \end{pgfonlayer}
\end{tikzpicture}
\caption{The descending blue ladder has branching\\ offshoots with game values $\alpha_1$, $\alpha_2$, $\alpha_3$, and so on}
\label{Figure.Infinite-branching-ladder}
\end{figure}
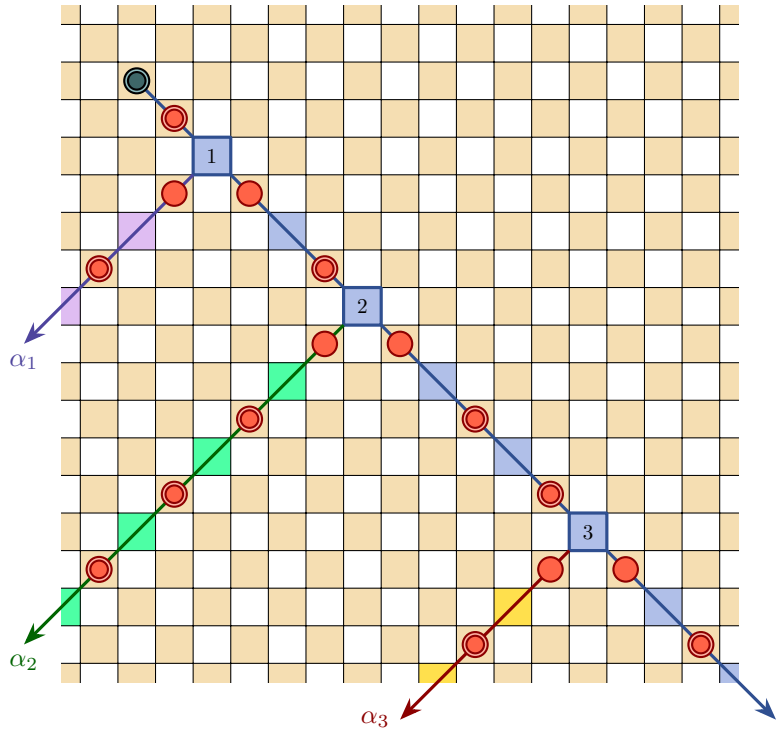

The branching squares on the main ladder can be spaced increasingly far apart, as desired, so as to make room for the branching of the prior trees to spread out. Since the entire binary branching tree fits into the plane, we can definitely fit copies of our desired branching trees. For the purpose of the game value, the only thing that matters about our position is the tree structure of the branches.

We claim, first, that this position overall is a winning position for Red under the forced-jump rule. Since Black is faced initially with a jumping opportunity, he is obliged to make a jump, and so he will begin to descend the main ladder. If he makes an infinite iterated jump, however, his only piece will disappear and he will lose immediately; and so he will instead want to stop at a safe square, such as one of the branching nodes on the main ladder. If he should stop at square $n$ on the main ladder, then Red can advance the main ladder pawn directly below and to the right of him, which would obligate Black to enter the offshoot tree. This would be exactly like making a first move in that tree, which had a game value of $\alpha_n$ for Red, when Black moves first, and so Red can ultimately force a win once he does so. If alternatively Black should on his initial move have already entered one of the offshoot subtrees, then his move amounts to having made a first move in that subtree, which again Red can win, since it has game value $\alpha_n$ for Red. So in any case, regardless of Black's initial move, Red can force a win.

A closer look at this analysis shows that the game value of the position overall is precisely $\sup_n(\alpha_n+1)$. By induction, for Black to enter the tree below square $n$ has game value $\alpha_n$ for Red. On his first move, he can position himself on square $n$, and if Red does not somehow corrupt that particular offshoot tree, then he can enter it on his next move, preserving at least game value $\alpha_n$ by doing so. If Red corrupts that particular offshoot tree, however, then Black can simply choose to descend on the main ladder to a further branching node realizing at least that same game value of $\alpha_n$---this was precisely why we had wanted the values to be nondecreasing. Red cannot afford always to allow Black to descend in this way on the main ladder, however, since this would mean infinite play and hence a draw, whereas Red aims to win. So eventually, Red will allow Black to enter a pristine offshoot subtree with value at least $\alpha_n$. (And furthermore, to allow Black to descend again on the main ladder is always simply a wasted move for Red, since it does not reduce the game value.) Therefore, the value of the position after Black's first move was at least $\alpha_n+1$, since we must count Red's move just before Black enters the subtree, and so the overall position has value $\sup_n(\alpha_n+1)$, as desired.

In summary, the main optimal line of play is that Black will occupy some square $n$. This will have value $\alpha_n+1$, because Red will move to block further descent on the main ladder by advancing the main ladder pawn below Black, and this forces Black to enter the subtree, which has value $\alpha_n$. It follows, we claim, that every countable ordinal is realized by such a tree-like position. As we said, values $0$ and $1$ are realized, and if value $\alpha$ is realized, then let every $\alpha_n=\alpha$ and construct the position of figure \ref{Figure.Infinite-branching-ladder} to realize value $\alpha+1=\sup_n(\alpha_n+1)$. If every ordinal below a countable limit ordinal $\lambda$ is realized, then choose $\alpha_n$ increasing and cofinal in $\lambda$,
so that we get a position with value $\sup_n(\alpha_n+1)=\lambda$, as desired.
\QEDbox\medskip\goodbreak 

An essentially similar idea can be used also to prove the following theorem.

\begin{theorem}\label{Theorem.No-forced-jump}
Every countable ordinal arises as the game value of a position in the version of infinite draughts with neither the forced-jump nor forced-iterated-jump rules.
\end{theorem}

\begin{wrapfigure}{r}{.58\textwidth}\hfill
\begin{tikzpicture}[scale=.5] 
 \def\n{12} size of board n= number of cols
 \def\m{18} 
 \useasboundingbox (-1,5) rectangle (\n,\m);
 \clip (-2,6) rectangle (\n,\m);
 \begin{scope}[shift={(-.5,-.5)}] 
  \begin{pgfonlayer}{boardshades} 
   \clip (-1.5,6.5) rectangle (\n+.5,\m+.5);
   \foreach \i in {-2,...,\n} {
    \foreach \j in {-1,...,\m} {
      \pgfmathsetmacro\c{100*abs(mod(\i+\j,2))};
      \fill [white!\c!Wheat] (\i,\j) rectangle (\i+1,\j+1);
      }
     }
  \end{pgfonlayer}
  \begin{pgfonlayer}{boardgrid} 
    \clip (-1.5,6.5) rectangle (\n+.5,\m+.5);
    \draw[thin] (-2,0) grid (\n+1,\m+1);
  \end{pgfonlayer}
 \end{scope}
 \draw (2,16) node[BB] (BB) {};
 \foreach \p/\q in {3/15,7/11,11/7,13/5,17/1,
  1/11,-1/9,
  5/7,3/5,1/3,
  11/1,
  4/16,0/16,8/12
 }
   {\draw (\p,\q) node[RR] {};}
 \foreach \p/\q in {5/13,9/9,15/3,
  3/13,
  7/9,
    13/3
 }
   {\draw (\p,\q) node[rr] {};}
 \begin{pgfonlayer}{boardshades}
   \foreach \p/\q in {4/14,6/12,8/10,10/8,12/6,14/4,16/2,18/0}
     {\draw (\p,\q) node[Square,fill=CadetBlue!60!Blue!40] {}; }
   \foreach \p/\q in {2/12,0/10,-2/8}
     {\draw (\p,\q) node[Square,fill=Orchid!80!Blue!40] {}; }
   \foreach \p/\q in {6/8,4/6,2/4,0/2}
     {\draw (\p,\q) node[Square,fill=SpringGreen!70] {}; }
   \foreach \p/\q in {12/2,10/0}
     {\draw (\p,\q) node[Square,fill=Yellow!50!Orange!70] {}; }
   \foreach \p/\q in {  4/16,0/16,8/12}
     {\draw (\p,\q) node[Square,fill=Yellow] {}; }
  \end{pgfonlayer}
  \begin{pgfonlayer}{boardarrows}
  \draw[CadetBlue!50!Navy] (4,14) node[Square,draw=CadetBlue!50!Navy,very thick] (1) {} node[black,scale=.8] {1}
        (8,10) node[Square,draw,very thick] (2) {} node[black,scale=.8] {2};
    \draw[CadetBlue!50!Navy,very thick,->] (BB) -- (1) -- (2) -- ++(5,-5);
    \draw[Orchid!50!CadetBlue!50!Navy,very thick,->] (1) -- ++(-6.7,-6.7) ;
    \draw[DarkGreen,very thick,->] (2) -- ++(-5,-5);
  \end{pgfonlayer}
\end{tikzpicture}
\captionsetup{style=rightside,font=footnotesize}
\caption{Red guardians (on the yellow squares)\\ can ensure
that Black stays on the tree}
\label{Figure.Tree-with-guardians}
\end{wrapfigure}
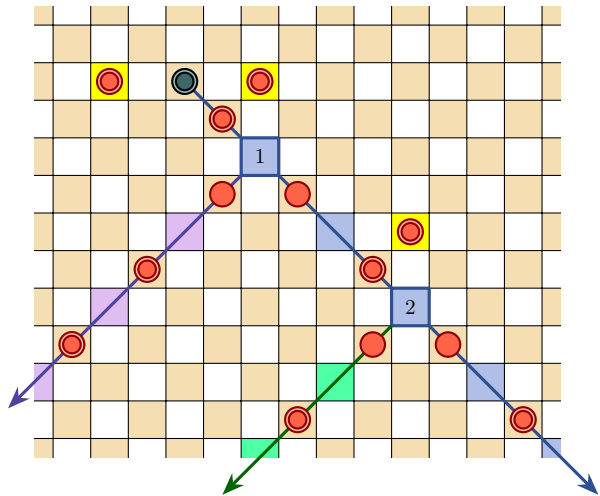
\smallskip\noindent\emph{Proof.} 
To prove this theorem we simply place some red guardians at strategic locations to keep Black on track. Here we have modified the previous position by placing two red guardian kings on the yellow highlighted squares near the black king, sufficient to encourage him to descend on the ladder even without the obligation of the forced-jump rule. Black would face immediate capture by the guardians should he choose not to descend the ladder, and so we don't need the forced-jump rule to convince him to do so. Similarly, we place a red guardian above each of the safe-haven branching nodes, above every square $n$ as indicated on the yellow hightlighted squares above squares $1$ and $2$, to ensure that after his brief rest, Black will indeed choose to continue his descent, either on the main ladder, if this is possible, or else on the offshoot subtree, rather than simply wandering off upward. With these changes, the situation is that Red can force Black to descend the tree (or die immediately), and so Red has a winning strategy in this position. But further, Black can choose to descend the tree as he likes, stopping at any desired branching square $n$, just as before, descending into the subtree only when it remains pristine and free from any corrupting Red movements. Thus, Black can play the position just as before, and so once again, every countable ordinal arises as the game value of a position in infinite draughts without the forced-jump rules.
\QEDbox\medskip\goodbreak 

We should like now to prove the main result, theorem \ref{MainTheorem}, establishing that every countable ordinal arises in infinite draughts, with the full rule set including the forced-jump and forced-iterated-jump rules.

\begin{proof}[Proof of theorem~\ref{MainTheorem}]
The main strategy of the proof will be the same as in theorems \ref{Theorem.No-forced-iteration} and \ref{Theorem.No-forced-jump}, but we shall have to work a little harder, in light of the forced-iterated-jump rule, in order to provide a genuinely safe resting place for Black. The position will therefore involve a somewhat more complicated branching unit configuration,

The main issue with the previous arguments is that the resting places used at the main branching nodes in the positions of the proofs of theorems \ref{Theorem.No-forced-iteration} and \ref{Theorem.No-forced-jump} are not suitable stopping points for Black when the forced-iterated-jump rule is in effect, since that rule would require him to continue jumping right past those squares without stopping. What is needed instead is a branching configuration with a resting square that is truly a place of respite, an end of the iterated jumping sequence that brought Black to that juncture.

We shall consider specifically a tree of paths connected by the unit branching configuration shown in figure \ref{Figure.Unit-branching-configuration}.
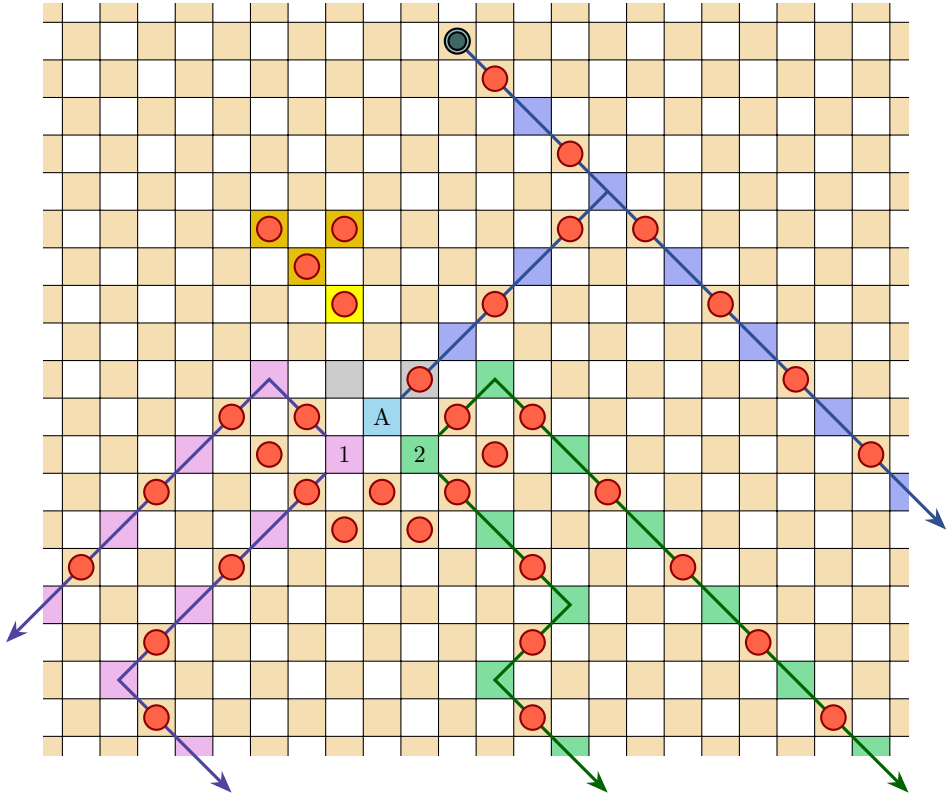
\begin{figure}[h]\centering
\begin{tikzpicture}[scale=.5] 
 \def\n{24} size of board n= number of cols
 \def\m{31} 
 \useasboundingbox (0,10) rectangle (\n+1,\m);
 \clip (1,11) rectangle (\n,\m);
 \begin{scope}[shift={(-.5,-.5)}] 
  \begin{pgfonlayer}{boardshades} 
   \clip (1.5,11.5) rectangle (\n+.5,\m+.5);
   \foreach \i in {-1,...,\n} {
    \foreach \j in {-1,...,\m} {
      \pgfmathsetmacro\c{100*abs(mod(\i+\j,2))};
      \fill [white!\c!Wheat] (\i,\j) rectangle (\i+1,\j+1);
      }
     }
  \end{pgfonlayer}
  \begin{pgfonlayer}{boardgrid} 
    \clip (1.5,11.5) rectangle (\n+.5,\m+.5);
    \draw[thin] (0,0) grid (\n+1,\m+1);
  \end{pgfonlayer}
 \end{scope}
 \draw (12,30) node[BB] (BB) {};
 \foreach \p/\q in {10/18,9/17,11/17,
  8/18,6/16,4/14,4/12,
  12/18,14/16,14/14,14/12,
  8/20,6/20,4/18,2/16,
  12/20,14/20,16/18,18/16,20/14,22/12,
  11/21,13/23,15/25,
  13/29,15/27,17/25,19/23,21/21,23/19,25/17,27/15,
  7/19,13/19,
  9/23,8/24,9/25,7/25
 }
   {\draw (\p,\q) node[rr] {};}
 \begin{pgfonlayer}{boardshades}
  \clip (0,0) rectangle (\n,\m);
   \foreach \p/\q in {12/22,14/24,14/28,16/26,18/24,20/22,22/20,24/18}
     {\draw (\p,\q) node[Square,fill=CadetBlue!30!Blue!40] {}; }
   \foreach \p/\q in {13/17,15/15,13/13,15/11,
     13/21,15/19,17/17,19/15,21/13,23/11}
     {\draw (\p,\q) node[Square,fill=SpringGreen!50!Green!50] {}; }
   \foreach \p/\q in {7/17,5/15,3/13,5/11,
      7/21,5/19,3/17,1/15}
     {\draw (\p,\q) node[Square,fill=Orchid!50] {}; }
   \foreach \p/\q in {9/21,11/21}
     {\draw (\p,\q) node[Square,fill=Gray!40] {}; }
   \foreach \p/\q in {8/24,9/25,7/25}
     {\draw (\p,\q) node[Square,fill=Yellow!70!Brown] {}; }
  \draw (9,23) node[Square,fill=Yellow] {};
   \draw (10,20) node[Square,fill=SkyBlue!80] (A) {} node[scale=.9] {A}
         (9,19) node[Square,fill=Orchid!50] (1) {} node[scale=.9] {1}
         (11,19) node[Square,fill=SpringGreen!50!Green!50] (2) {} node[scale=.9] {2};
  \end{pgfonlayer}
  \begin{pgfonlayer}{boardarrows}
    \draw[CadetBlue!50!Navy,very thick,->] (A) -- ++(6,6) (BB) -- ++(13,-13);
    \draw[Orchid!50!CadetBlue!50!Navy,very thick,->] (1) -- ++(-2,2) -- ++(-7,-7);
    \draw[Orchid!50!CadetBlue!50!Navy,very thick,->] (1) -- ++(-6,-6) -- ++(3,-3);
    \draw[DarkGreen,very thick,->] (2) -- ++(2,2) -- ++(11,-11);
    \draw[DarkGreen,very thick,->] (2) -- ++(4,-4) -- ++(-2,-2) -- ++(3,-3);
  \end{pgfonlayer}
\end{tikzpicture}
\caption{The unit branching configuration}
\label{Figure.Unit-branching-configuration}
\end{figure}
The pattern of play is that Black will have been proceeding down on the main ladder, in blue, which will have perhaps many of these unit branching configurations appearing along it with various game values. In order to avoid loss by infinite iterated jump, Black will turn off the main ladder into one of these unit configurations, whichever one he prefers, and come to a rest at square A, an end to the iterated-jumping line of this turn. The red pieces that he had jumped in arriving at square A will have of course been removed from the board. From square A, Black aims to follow one of the violet or green exit lines leading out of the configuration, and indeed, he aims to follow one of these exit lines which remains pristine, in the sense that it and the entire subtree to which it leads to has not yet experienced any corrupting Red moves. We shall argue that indeed it is possible for him to do so. It might help to imagine that after departing the part of the position shown here, these four exit lines begin to separate from one another and then branch apart at great distance before encountering the next unit branching configuration upon them.

The default normal play we have in mind is that after Black arrives at square A, Red will make a move (perhaps within this configuration or perhaps somewhere else completely), and then Black will proceed either to square 1 or square 2, so as to aim at an exit path that remains uncorrupted by the immediately preceding Red move. From either of these squares, he is able to choose between two exit lines leading out of the configuration. If on her next move Red should corrupt one of those two exit lines, then Black will be able to leave the configuration on the other. This exit line then serves as the new main ladder leading to the next unit configuration.

Notice that Black will want to move from square A either to square 1 or square 2, because either of the moves upward to the gray highlighted squares places him in a position to be quickly captured, as Red will advance the leading red guardian highlighted in yellow and obligate a forced-jump reply after which Black is captured. Because of this, Red can ensure that Black does in fact move either to square 1 or square 2 (unless the exit lines already have a trivial game value). And if Red has made no corrupting moves nearby, then Black is indeed obligated to exit on one of the four colored exit lines.

It follows that the position overall will be a win for Red, since the connectivity of the unit branching configurations is such that the exit lines from the $n$th branching unit each lead to a tree-like position with game value $\alpha_n$, as in theorems \ref{Theorem.No-forced-iteration} and \ref{Theorem.No-forced-jump}, and so Red will have a winning strategy once Black exits the unit.

The question now is whether Red can achieve a quicker win, that is, with a lower game value than forcing Black to descend the tree as we described. We shall argue that this is not the case, in fact Black can succeed in his plan to exit the configuration on a pristine exit line and thereby descend through the abstract tree represented by the connectivity of these unit configurations. The reason is that when Black arrives at square A, we may assume that all four exit lines and the rest of this unit configuration remain pristine, since Black would have chosen this particular offshoot from the main ladder only in the case that this was true. After Black has arrived at square A, there is no immediate move by Red with a forced-jump reply. There are four exit lines proceeding away from squares 1 and 2, and no single Red move can disrupt more than one of these exit lines. Thus, regardless of Red's move, Black will be able to move either to square 1 or to square 2, whichever has both of its exit lines remaining pristine. And then Red makes another move, which again disrupts at most one of those two exit lines, and thus Black is able to depart on the other one. Thus, from the initial position, Black will be able to choose an acceptable unit branching configuration on the main ladder and subsequently exit on a pristine exit line.

We may now complete the proof of the theorem by induction. Suppose that there are such tree-like positions having game value $\alpha_n$ for every natural number $n$. We may assemble a main ladder having infinitely many branching unit configurations, with all four exit lines of the $n$th such unit leading to such positions of value~$\alpha_n$. The resulting position, we claim, with Black ready to start descending that ladder, will have game value at least $\sup_n(\alpha_n+2)$. First of all, the position is a definite win for Red, since Red can choose always to never interfere with the default normal play, and so Black will either jump to infinity and lose immediately, or else come to rest at some square A for one of the accessible configurations. From there, Black will not move to a gray square, since this causes a quick loss, and so he will proceed either to square 1 or 2, after which he is on the ladder of a position with game value $\alpha_n$ for some $n$, a win for Red. So Red can win our overall position, and the only question is the game value. We have argued, however, that for any $n$ Black can situate himself after one supplemental move on one of the exit lines of configuration with value $\alpha_n$. So our position has value at least $\alpha_n+2$, making the overall value at least $\sup_n(\alpha_n+2)$. From this, it follows as in the proof of theorem \ref{Theorem.No-forced-iteration} that the game values realized in positions of infinite draughts are unbounded in the countable ordinals. And since it is easy to prove by induction that the game values that are realized are closed downward---if a game position has value $\beta$ and $\alpha<\beta$, then there is a position reachable in finitely many steps having value $\alpha$. So every countable ordinal is realized as the game value of a position in infinite draughts.
\end{proof}

Since every position in infinite draughts admits a choice of at most countably many moves, it follows by induction that the  game values realized in infinite draughts will be countable ordinals. And so another way of summarizing what we have proved is that every position in infinite draughts is either a draw for both players, or else a win for one player or the other with some countable ordinal game value, with every countable ordinal being realized this way. In particular, the \emph{omega one} of infinite draughts, the supremum of the game values realized by positions in this game, will be true $\omega_1$:

$$\omega_1^{\text{\tiny draughts}}=\omega_1.$$

Let us conclude the article with a brief consideration of the role of computable strategies in infinite draughts. We say that a position in infinite draughts is \emph{computable} if there is a computable procedure to identify for each square whether it is occupied or not and if so by what kind of piece. So there is a computable procedure to produce an image of any desired finite part of the position.\goodbreak

\begin{theorem}
There is a computable position in infinite draughts, such that Red has a computable strategy that wins against any computable Black strategy, and forces a draw or better against any Black strategy. Meanwhile, Black has a (noncomputable) drawing strategy.
\end{theorem}

Thus, the game is a draw, but when the players must play according to computable strategies, then Red can force a win.

\begin{proof}
The proof follows the same strategy as \cite[theorem~6]{EvansHamkins2014:TransfiniteGameValuesInInfiniteChess}. Namely, it is a classical result of computability theory that there is an infinite computable binary branching tree $T$ with no computable infinite branch. Using the methods of theorems \ref{MainTheorem}, \ref{Theorem.No-forced-iteration}, and \ref{Theorem.No-forced-jump}, we can create a position in infinite draughts so that play unfolds as though Black is climbing through $T$. Since $T$ admits an infinite branch, there will be a way for Black to climb the tree without ever getting stuck in a terminal node, and this will be a draw in the infinite draughts position by infinite play. But if Black plays according to a computable strategy, then he will not be able to choose at the branching nodes in such a way that will follow an infinite branch, and so inevitably he will find himself stuck at a terminal node, where he will lose. The strategy for Red in either case is to implement the moves that force Black to keep climbing the tree, in the manner discussed in those previous theorems.
\end{proof}

The proof works whether or not we have the forced-jump and forced-iterated-jump rules, and so the theorem applies to all three rule sets.

\printbibliography


\end{document}

--------------------------------------------------------------------

\section*{Graveyard}

This section contains some text I've copied here from the old versions of the arguments, in case we want to revert.

From the old proof of theorem \ref{Theorem.No-forced-iteration}:

\begin{wrapfigure}{r}{.6\textwidth}\hfill
\begin{tikzpicture}[scale=.5] 
 \def\n{13} size of board n= number of cols
 \def\m{17} 
 \clip (-1,-1) rectangle (\n,\m);
 \begin{scope}[shift={(-.5,-.5)}] 
  \begin{pgfonlayer}{boardshades} 
   \clip (-.5,-.5) rectangle (\n+.5,\m+.5);
   \foreach \i in {-1,...,\n} {
    \foreach \j in {-1,...,\m} {
      \pgfmathsetmacro\c{100*abs(mod(\i+\j,2))};
      \fill [white!\c!Wheat] (\i,\j) rectangle (\i+1,\j+1);
      }
     }
  \end{pgfonlayer}
  \begin{pgfonlayer}{boardgrid} 
    \clip (-.5,-.5) rectangle (\n+.5,\m+.5);
    \draw[thin] (-1,-1) grid (\n+1,\m+1);
  \end{pgfonlayer}
 \end{scope}
 \draw (7,1) node[BB] {};
 \foreach \p/\q in {8/2,10/4,10/6,8/8,
 6/8,4/8,2/10,
 8/10,8/12,6/14
 }
   {\draw (\p,\q) node[rr] {};}
 \begin{pgfonlayer}{boardshades}
  \clip (0,0) rectangle (\n,\m);
   \foreach \p/\q in {9/3,11/5,9/7,7/9}
     {\draw (\p,\q) node[Square,fill=CadetBlue!60!Blue!40] {}; }
   \foreach \p/\q in {9/11,7/13,5/15,
    5/7,3/9,1/11
     }
     {\draw (\p,\q) node[Square,fill=Orchid!80!Blue!40] {}; }
  \draw (7,9) node[Square] (A) {} node[scale=.8] {A}
        (5,15) node[Square] (b) {} node {$b$}
        (1,11) node[Square] (a) {} node {$a$};
  \end{pgfonlayer}
  \begin{pgfonlayer}{boardarrows}
    \draw[CadetBlue!50!Navy,very thick] (7,1) -- (11,5) -- (A);
    \draw[Orchid!50!CadetBlue!50!Navy,very thick] (A) -- ++(2,2) -- (b);
    \draw[Orchid!50!CadetBlue!50!Navy,very thick] (A) -- ++(-2,-2) -- (a);
  \end{pgfonlayer}
\end{tikzpicture}
\end{wrapfigure}
It can often be dangerous for a draughts player to halt their move in the middle of an iterated jump, since in many cases one would be faced with immediate recapture. Nevertheless, in some positions, such as the one shown here, one may find a temporary safe haven in the middle of an iterated jump. In this position, Black may undertake an iterated jump from the bottom, climbing up to square A, and although he could continue jumping past this square, it would also be safe to stop there for a brief respite, since he would face no immediate threat, being ensconced safely between the red pieces. On his next turn, Black could then choose either to travel on the branch toward square~$a$ or alternatively on the other branch, toward square~$b$.

This configuration will serve as the unit branching element in more elaborate positions that will establish the theorem. Namely, suppose that we have tree-like positions with game value $\alpha_n$ for every $n$, in the version of draughts without the forced-iterated-jump rule. We may now simply combine them together to produce a position with value $\sup_n(\alpha_n+1)$ as in the position of figure \ref{Figure.Ladder-with-infinite-choices}, where square $n$ leads to two branches, each of which leads to a tree-like position (not shown) where the game value is $\alpha_n$.

\begin{figure}[h]
\centering
\begin{tikzpicture}[scale=.4] 
 \def\n{29} size of board n= number of cols
 \def\m{33} 
 \clip (-1,-1) rectangle (\n,\m);
 \begin{scope}[shift={(-.5,-.5)}] 
  \begin{pgfonlayer}{boardshades} 
   \clip (-.5,-.5) rectangle (\n+.5,\m+.5);
   \foreach \i in {-1,...,\n} {
    \foreach \j in {-1,...,\m} {
      \pgfmathsetmacro\c{100*abs(mod(\i+\j,2))};
      \fill [white!\c!Wheat] (\i,\j) rectangle (\i+1,\j+1);
      }
     }
  \end{pgfonlayer}
  \begin{pgfonlayer}{boardgrid} 
    \clip (-.5,-.5) rectangle (\n+.5,\m+.5);
    \draw[thin] (-1,-1) grid (\n+1,\m+1);
  \end{pgfonlayer}
 \end{scope}
 \draw (7,1) node[BB] {};
 \foreach \p/\q in {8/2,10/4,12/6,14/8,16/10,18/12,20/14,22/16,24/18,26/20,28/22,30/24,
 10/6,8/8,
 6/8,4/8,2/10,
 8/10,8/12,6/14,
 18/14,16/16,
 14/16,12/16,10/18,
 16/18,16/20,14/22,
 26/22,24/24,
 22/24,20/24,18/26,
 24/26,24/28,22/30
  }
   {\draw (\p,\q) node[rr] {};}
 \begin{pgfonlayer}{boardshades}
  \clip (0,0) rectangle (\n,\m);
   \foreach \p/\q in {9/3,11/5,13/7,15/9,17/11,19/13,21/15,23/17,25/19,27/21,29/23,31/25,
   9/7,7/9,
   17/15,15/17,
   25/23,23/25}
     {\draw (\p,\q) node[Square,fill=CadetBlue!60!Blue!40] {}; }
   \foreach \p/\q in {9/11,7/13,5/15,
    5/7,3/9,1/11
     }
     {\draw (\p,\q) node[Square,fill=Orchid!80!Blue!40] {}; }
   \foreach \p/\q in {17/19,15/21,13/23,
    13/15,11/17,9/19
     }
     {\draw (\p,\q) node[Square,fill=SpringGreen!70] {}; }
   \foreach \p/\q in {25/27,23/29,21/31,
    21/23,19/25,17/27
     }
     {\draw (\p,\q) node[Square,fill=Orange!50!Yellow!80] {}; }
    \draw (7,9) node[Square] (1) {} node[scale=.8] {$1$}
        (5,15) node[Square] (b1) {} node[scale=.65] {$b_1$}
        (1,11) node[Square] (a1) {} node[scale=.7] {$a_1$}
        (15,17) node[Square] (2) {} node[scale=.8] {$2$}
        (13,23) node[Square] (b2) {} node[scale=.65] {$b_2$}
        (9,19) node[Square] (a2) {} node[scale=.7] {$a_2$}
        (23,25) node[Square] (3) {} node[scale=.8] {$3$}
        (21,31) node[Square] (b3) {} node[scale=.65] {$b_3$}
        (17,27) node[Square] (a3) {} node[scale=.7] {$a_3$};
  \end{pgfonlayer}
  \begin{pgfonlayer}{boardarrows}
    \draw[CadetBlue!50!Navy,very thick,->,shorten >=-6pt] (7,1) -- (\n+1,\n-5);
    \draw[CadetBlue!50!Navy,very thick] (11,5) -- (1);
    \draw[Orchid!50!CadetBlue!50!Navy,very thick] (1) -- ++(2,2) -- (b1);
    \draw[Orchid!50!CadetBlue!50!Navy,very thick] (1) -- ++(-2,-2) -- (a1);
    \draw[DarkGreen,very thick] (19,13) -- (2);
    \draw[DarkGreen,very thick] (2) -- ++(2,2) -- (b2);
    \draw[DarkGreen,very thick] (2) -- ++(-2,-2) -- (a2);
    \draw[RawSienna,very thick] (27,21) -- (3);
    \draw[RawSienna,very thick] (3) -- ++(2,2) -- (b3);
    \draw[RawSienna,very thick] (3) -- ++(-2,-2) -- (a3);
  \end{pgfonlayer}
\end{tikzpicture}
    \caption{A ladder with infinitely many offshoot choices}
    \label{Figure.Ladder-with-infinite-choices}
\end{figure}
The overall pattern of play is like this. At the outset, Black faces a forced-jump requirement. He can come to rest at any desired square $n$, and then Red will make a move. Perhaps that move might corrupt the position in the tree somehow, which will aim to trap Black by eventually organizing an army of red pieces through those corrupting moves. But the main point is that since there are two branches leading away from square $n$, either toward $a_n$ or $b_n$, at most one of those trees was corrupted in the red move, and so Black can opt to follow the pristine path. Since we may assume that the paths beyond $a_n$ and $b_n$ lead to another infinite branch of choices (unless we are already at a terminal node of the tree), and so the next resting place will be far above any previous corrupting moves by Red.

In this way, the main line of play will be that Black can in effect mimic the behavior of climbing through the tree that is represented by the position. At terminal nodes, Black will be captured and lose. Since the ordinal game values of the climbing-through-$T$ game reach arbitrarily high in the countable ordinals, it follows that we have produced positions in infinite draughts (without the forced-iterated-jump rule) with game values as high in the countable ordinals as desired. Since the game values are furthermore an initial segment of the ordinals, it follows that every countable ordinal is realized as the game value of a position in this version of infinite draughts.
\QEDbox

Old proof of theorem \ref{Theorem.No-forced-jump}:

\begin{wrapfigure}{r}{.5\textwidth}\hfill
\begin{tikzpicture}[scale=.5] 
 \def\n{13} size of board n= number of cols
 \def\m{19} 
 \clip (1,1) rectangle (\n,\m);
 \begin{scope}[shift={(-.5,-.5)}] 
  \begin{pgfonlayer}{boardshades} 
   \clip (1.5,1.5) rectangle (\n+.5,\m+.5);
   \foreach \i in {-1,...,\n} {
    \foreach \j in {-1,...,\m} {
      \pgfmathsetmacro\c{100*abs(mod(\i+\j,2))};
      \fill [white!\c!Wheat] (\i,\j) rectangle (\i+1,\j+1);
      }
     }
  \end{pgfonlayer}
  \begin{pgfonlayer}{boardgrid} 
    \clip (1.5,1.5) rectangle (\n+.5,\m+.5);
    \draw[thin] (-1,-1) grid (\n+1,\m+1);
  \end{pgfonlayer}
 \end{scope}
 \draw (7,5) node[BB] (BB) {};
 \foreach \p/\q in {8/6,10/8,12/10,
 10/10,8/12,
 6/12,4/12,2/14,
 8/14,8/16,6/18
 }
   {\draw (\p,\q) node[rr] {};}
 \foreach \p/\q in {7/3,5/7,
 5/13,7/11
 }
   {\draw (\p,\q) node[RR] {};}
 \begin{pgfonlayer}{boardshades}
  \clip (0,0) rectangle (\n,\m);
   \foreach \p/\q in {9/7,11/9,13/11,9/11,7/13}
     {\draw (\p,\q) node[Square,fill=CadetBlue!60!Blue!40] {}; }
   \foreach \p/\q in {9/15,7/17,5/19,
    5/11,3/13,1/15
     }
     {\draw (\p,\q) node[Square,fill=Orchid!80!Blue!40] {}; }
   \foreach \p/\q in {7/3,5/7}
     {\draw (\p,\q) node[Square,fill=Yellow!60!Orange] {}; }
   \foreach \p/\q in {5/13,7/11}
     {\draw (\p,\q) node[Square,fill=Yellow] {}; }
  \coordinate (A) at (7,13);
  \coordinate (b) at (5,19);
  \coordinate (a) at (1,15);
  \end{pgfonlayer}
  \begin{pgfonlayer}{boardarrows}
    \draw[CadetBlue!50!Navy,very thick,->] (BB) -- +(7,7);
    \draw[Orchid!50!CadetBlue!50!Navy,very thick,->,shorten >=-12pt] (11,9) -- (A) -- ++(2,2) -- (b);
    \draw[Orchid!50!CadetBlue!50!Navy,very thick,->,shorten >=-12pt] (A) -- ++(-2,-2) -- (a);
  \end{pgfonlayer}
\end{tikzpicture}
\end{wrapfigure}
\smallskip\noindent\emph{Proof.} 
For this, we simply place some red guardians at strategic locations to keep Black on track. Here we have modified the previous position by placing two red guardian kings on the orange highlighted squares near the king, sufficient to encourage him to climb the ladder even without the obligation of the forced-jump rule. Black would face immediate recapture by the guardians should he choose not to climb the ladder, and so we don't need the forced-jump rule to convince him to do so. Similarly, we have placed red guardians on the yellow highlighted squares near all the safe-haven branching nodes, to ensure that after his brief rest, Black will indeed choose to continue on one branch or the other, rather than simply wandering off. With these changes, the situation is that Red can force Black to climb the tree (or die immediately), and so Red has a winning strategy in this position. But further, Black can choose to climb the tree as he likes, using the infinitely branching node opportunities to continue into a pristine part of the tree, stopping for respite at the resting nodes, and continuing on with which of the branches remains pristine. Thus, Black can play the position as though climbing the tree, which has high game value. Therefore every countable ordinal arises as the game value of a position in infinite draughts without the forced-jump rules.
\QEDbox\medskip\goodbreak 

Old proof of the main theorem, theorem \ref{MainTheorem}.

\begin{proof}[Proof of theorem~\ref{MainTheorem}]
The main strategy of the proof will be the same as in theorems \ref{Theorem.No-forced-iteration} and \ref{Theorem.No-forced-jump}, but we shall have to work a little harder, in light of the forced-iterated-jump rule, in order to provide a genuinely safe resting place for Black. The position will therefore involve a somewhat more complicated branching unit configuration,

The main issue with the previous arguments is that the resting place used at the main branching node in the positions of the proofs of theorems \ref{Theorem.No-forced-iteration} and \ref{Theorem.No-forced-jump} is not a suitable stopping point for Black when the forced-iterated-jump rule is in effect, since that rule would require him to continue jumping right past that square without stopping. What is needed instead is a unit branching configuration with a resting square that is truly a place of respite, an end of the iterated jumping sequence that brought Black to that juncture.

The pattern of play will be that Black will pursue an iterated jumping path that brings him to such a resting square; Red will then make a move; Black will then be able either to depart immediately on a pristine iterated jumping path proceeding to the next node, or else he will be able to place himself into a position where he has the choice of two such paths; Red will move, and then Black will choose to follow whichever of the two iterated jumping paths remains pristine. In this way, the pattern of play for Black will amount to climbing in the tree whose tree structure was implemented by the connectivity of those branching units. Thus, it will have as high a game value as desired.

Consider specifically the unit branching configuration shown here in figure \ref{Figure.Main-branching-unit-configuration}.

\begin{figure}[h]
    \centering
\begin{tikzpicture}[scale=.4] 
 \def\n{33} size of board n= number of cols
 \def\m{34} 
 \clip (2,2) rectangle (\n,\m);
 \begin{scope}[shift={(-.5,-.5)}] 
  \begin{pgfonlayer}{boardshades} 
   \clip (2.5,2.5) rectangle (\n+.5,\m+.5);
   \foreach \i in {-1,...,\n} {
    \foreach \j in {-1,...,\m} {
      \pgfmathsetmacro\c{100*abs(mod(\i+\j,2))};
      \fill [white!\c!Wheat] (\i,\j) rectangle (\i+1,\j+1);
      }
     }
  \end{pgfonlayer}
  \begin{pgfonlayer}{boardgrid} 
    \clip (2.5,2.5) rectangle (\n+.5,\m+.5);
    \draw[thin] (0,0) grid (\n+1,\m+1);
  \end{pgfonlayer}
 \end{scope}
 \draw (5,5) node[BB] (BB) {};
 \foreach \p/\q in {6/6,8/8,10/10,12/12,14/14,16/16,
 16/18,16/20,16/22,
 18/20,20/18,22/18,24/20,26/22,28/24,30/26,32/28,34/30,
 15/21,17/19,19/17,19/15,17/13,15/11,13/9,11/9,9/11,7/13,5/15,3/17,1/19,
 13/23,11/25,9/27,7/29,5/31,3/33,
 17/23,19/21,21/21,23/23,25/25,27/27,27/29,25/31,23/33,
 15/25,13/27,11/29,9/31,7/33,5/35,
 13/25,12/26,10/28,8/30,6/32,4/34,
 12/22,10/24,8/26,6/28,4/30,2/32,0/34,
 16/26,14/28,12/30,10/32,8/34
  }
   {\draw (\p,\q) node[rr] {};}
 \begin{pgfonlayer}{boardshades}
  \clip (2,2) rectangle (\n,\m);
   \foreach \p/\q in {7/7,9/9,11/11,13/13,15/15,17/17,
   15/19,17/21,19/19,
   21/17,23/19,25/21,27/23,29/25,31/27,33/29}
     {\draw (\p,\q) node[Square,fill=CadetBlue!60!Blue!40] {}; }
   \draw (15,23) node[Square,fill=Orchid!50] (A) {} node[scale=.8] {A};
   \draw (A) + (1,1) node[Square,fill=Green!30] (2) {} node[scale=.7] {2}
         (A)+(-1,-1) node[Square,fill=Green!30] (1) {} node[scale=.7] {1};
   \draw (A) + (2,2) node[Square,fill=Orange!80] (b) {} node[scale=.7] {b}
         (A) + (-2,-2) node[Square,fill=Orange!80] (a) {} node[scale=.7] {a}
         (A) + (-1,1) node[Square,fill=Orange!80] (c) {} node[scale=.7] {c};
   \foreach \p/\q in {16/20,18/18,20/16,18/14,16/12,14/10,12/8,10/10,8/12,6/14,4/16,2/18,0/20,
    12/24,10/26,8/28,6/30,4/32,2/34
     }
     {\draw (\p,\q) node[Square,fill=SpringGreen!60] {}; }
   \foreach \p/\q in {18/22,20/20,22/22,24/24,26/26,28/28,26/30,24/32,22/34,
    14/26,12/28,10/30,8/32,6/34
     }
     {\draw (\p,\q) node[Square,fill=SpringGreen!60] {}; }
   \foreach \p/\q in {13/25,11/27,9/29,7/31,5/33,
     15/27,13/29,11/31,9/33,
     11/23,9/25,7/27,5/29,3/31,1/33
     }
     {\draw (\p,\q) node[Square,fill=Yellow!65!Orange] {}; }
  \end{pgfonlayer}
  \begin{pgfonlayer}{boardarrows}
    \draw[CadetBlue!50!Navy,ultra thick,->] (BB) -- ++(12,12) -- ++(-2,2) -- ++(2,2) -- (A) -- ++(6,-6) -- (\n+1.5,\n-2.5);
    \draw[SpringGreen!30!DarkGreen,very thick,->] (1) -- ++(6,-6) -- ++(-8,-8) -- ++(-11,11);
    \draw[SpringGreen!30!DarkGreen,very thick,->] (1) -- ++(-13,13);
    \draw[SpringGreen!30!DarkGreen,very thick,->] (2) -- ++(4,-4) -- ++(8,8) -- ++(-7,7);
    \draw[SpringGreen!30!DarkGreen,very thick,->] (2) -- ++(-11,11);
    \draw[Brown,very thick,->] (b) -- ++(-10,10);
    \draw[Brown,very thick,->] (a) -- ++(-12,12);
    \draw[Brown,very thick,->] (c) -- ++(-11,11);
 \end{pgfonlayer}
\end{tikzpicture}
    \caption{Main branching unit configuration}
    \label{Figure.Main-branching-unit-configuration}
\end{figure}

Black will enter the configuration from below by climbing on the main ladder, in blue. The first thing to notice is that Black can choose simply to bypass the configuration entirely by jumping right through it, following the twists and turns of the main ladder in order to exit at the upper right. By doing so, Black would then simply proceed to the next such unit configuration appearing on on the main ladder, which will have infinitely many such configurations appearing in succession. In this way, Black can in effect choose to interact with whichever configuration on the main ladder that he desires, and it is exactly this feature that makes the current Black position represent abstractly an infinitely branching node---on this turn Black can choose from amongst the infinitely many such branching unit configurations that are available, just like choosing amongst the successor nodes in the tree that the entire position will be representing.

Although it would be legal for Black to carry out an infinite iterated jump on this turn, climbing the main ladder to infinity, doing so would cause his piece to disappear in light of the infinite-iterated-jump rule, leading to an immediate loss. Black will therefore choose instead to engage with one of the branching unit configurations appearing along the main ladder by turning left off the main ladder and coming to rest at square A, a true resting square, where his iterated jump on this turn would come to an end.

So let us consider the situation after Black has climbed the main ladder and come to rest at square A. The red pieces that he had jumped to arrive there of course have now been removed. From square A, Black aims to follow one of the green or orange exit lines leading out of this configuration, and indeed, he aims to follow one of these exit paths which remains pristine, uncorrupted by whatever moves Red makes while Black interacts with this configuration. We shall argue that indeed it is possible for him to do so. It might help to imagine that after departing the part of the position shown here, these seven exit lines begin to separate from one another and then branch apart at great distance before encountering the next unit branching configuration upon them.

The default normal play we have in mind is that Red will make a move (perhaps somewhere else), and then Black can proceed either to square 1 or square 2, so as to aim at at exit path that remains uncorrupted by the immediately preceding Red move. From either of these squares, he is able to choose between two green paths leading out of the configuration. If on his next move Red should corrupt one of those two paths, then Black can leave the configuration on the other. This green exit path then serves as the new main ladder leading to the next unit configuration.

Let us argue that indeed Red can force this line of play and thereby ultimately win. Namely, after Black arrives at square A, Red can advance the red piece directly south of square A further south. This makes both square 1 and the square south east of square A inhospitable for Black, who would face immediate capture there; and since square c is similarly inhospitable, Black will advance to square 2, after which he faces an exit on one of the two green exit lines leading out of the position. Thus, Red can force Black to follow one of the exit lines and thereby to climb the abstract tree that the is represented by the connectivity of the branching node configuration units. Since that tree has no infinite branches, eventually Red will force Black into a terminal node situation, where Red will win the game. Thus, overall the position is ultimately a win for Red.

The question now is whether Red can achieve a quicker win, that is, with a lower game value than forcing Black to climb the tree. We shall argue that no, in fact Black can succeed in his plan to exit the configuration on a pristine exit line and thereby climb the abstract tree represented by the connectivity of these unit configurations.

There are a variety of ways that Red might try to interfere with the default normal play, perhaps in attempt to secure a quicker win. So let us consider the possibilities. With Black on square A, it is possible that Red might advance a red piece directly southeast, on the now-empty square below the Black king. In this event, Black is obligated to jump, which puts him back on the main ladder leading up to succeeding unit configuration. Red has in effect kicked Black out of this particular branching unit configuration, pushing him back onto the main ladder. We shall assume that on every main ladder line, there are infinitely many branching unit configurations that represent the same abstract node in the tree we are implementing in this position, and so Black is perfectly happy simply to proceed to the next such desired node on this same main ladder. Notice, however, that Red cannot afford always to kick Black out of the position in this way, since this would cause infinitely long play, which is a draw, but Red is aiming to win. Eventually, therefore, Red will not make this move behind Black on the main ladder.

Red might alternatively attempt to interfere with the normal play by moving a piece onto square 1. In this case, Black is obligated to jump, but indeed he is happy to do so, since he can then simply proceed with iterative jumping onto the orange exit line, fulfilling his aim to exit on one of the green or orange exit lines. This line will be pristine, since the Red move to square 1 could not have disturbed the orange exit line. An exactly similar situation arises if Red should instead move a piece onto square 2, since Black would jump it and then exit on the upper orange line.

Red might attempt to interfere by advancing his piece from the central orange line to square c. This would obligate Black to jump, but again, he would be happy to do so, proceeding to exit on that central orange exit line.

So we've considered the four squares adjacent to square A and how Black would react if Red would occupy one of them, obligating an immediate Black jumping response, and in each case, Black is entirely happy with the resulting line of play, either by following a pristine exit line out of the configuration or else getting back onto the main ladder.

How else might Red interfere with the position? The main observation to make is that if we assume that the configuration was pristine when Black came to rest at square A, as well as the rest of the tree of further such configurations reached by the seven exit lines---Black would choose to interact with this configuration only in the case that this is true---then Red has only one move to change things. This move, therefore, can have affected only one of the four green exit lines. So Black can choose to move to square 1 or 2 as in the default normal line of play, placing himself in a position to choose between two pristine green exit lines. On the next move, Red can ruin only one of these, and so Black can choose to exit on the other.

Thus, we have argued altogether that from the initial position, Black will be able to choose an acceptable unit branching configuration on the main ladder and subsequently exit on a pristine exit line.

We may now complete the proof of the theorem by induction. Suppose that there are such tree-like positions having game value $\alpha_n$ for every natural number $n$. We may assemble a single main line having infinitely many unit configurations, with all seven exit lines leading to such positions of value $\alpha_n$, with infinitely many for each $n$. The resulting position, we claim, with Black ready to start climbing that ladder, will have game value at least $\sup_n\alpha_n+1$. First of all, the position is a definite win for Red, since Red can choose always to never interfere with the default normal play, and so Black will either jump to infinity and lose immediately, or else come to rest at some square A for one of the accessible configurations. From there, Black moving to square c is an immediately loss, and so he will have to proceed either to square 1 or 2, from which he is on the ladder of a position with game value $\alpha_n$ for some $n$, a win for Red. So Red can win our overall position, and the only question is the game value. We have argued, however, that for any $n$ Black can situate himself after one or two moves on one of the exit lines of a $\alpha_n$ value configuration. So our position has value at least $\alpha_n+1$, making the overall value at least $\sup_n\alpha_n+1$. (In fact, the value will be exactly $\sup_n\alpha_n+2$, since Red can make Black take two moves for each unit configuration---one move to enter it, then another to square 1 or square 2, and then finally leaving, which is the same as the entering move for the unit configuration at the next level.)

From this, it follows that the game values realized in positions of infinite draughts are unbounded in the countable ordinals. It is easy to prove by induction that the game values that are realized are closed downward---if a game position has value $\beta$ and $\alpha<\beta$, then there is a position reachable in finitely many steps having value $\alpha$. So every countable ordinal is realized as the game value of a position in infinite draughts.
\end{proof}

%% file: MathMacrosJDH.tex
\newtheorem{theorem}{Theorem}
\newtheorem*{theorem*}{Theorem}

\newtheorem*{maintheorem*}{Main Theorem}
\newtheorem*{maintheorems*}{Main Theorems}

\newtheorem*{corollary*}{Corollary}
\newtheorem*{corollaries*}{Corollaries}

\theoremstyle{definition}
\newtheorem{definition}[theorem]{Definition}
\newtheorem*{definition*}{Definition}

\newtheorem*{question*}{Question}

\newtheorem*{questions*}{Questions}
\newtheorem*{mainquestion*}{Main Question} 
\newtheorem*{openquestion*}{Open Question} 
\theoremstyle{remark}

\newcommand{\QED}{\end{proof}}

\def\proclaim[#1]{{\bf #1}}
\def\BF#1.{{\bf #1.}}

\def\says#1:#2\par{\item[#1] #2\par}

\newcommand\p{\frak{p}}

\makeatletter
\newcommand{\dotminus}{\mathbin{\text{\@dotminus}}}
\newcommand{\@dotminus}{%
  \ooalign{\hidewidth\raise1ex\hbox{.}\hidewidth\cr$\m@th-$\cr}%
}
\makeatother

\newcommand{\set}[1]{\{\,{#1}\,\}}

\renewcommand{\setminus}{\raise.3ex\hbox{\rotatebox{-20}{$-$}}} 

\newcommand{\smalllt}{\mathrel{\mathchoice{\raise2pt\hbox{$\scriptstyle<$}}{\raise1pt\hbox{$\scriptstyle<$}}{\raise0pt\hbox{$\scriptscriptstyle<$}}{\scriptscriptstyle<}}}
\newcommand{\smallleq}{\mathrel{\mathchoice{\raise2pt\hbox{$\scriptstyle\leq$}}{\raise1pt\hbox{$\scriptstyle\leq$}}{\raise1pt\hbox{$\scriptscriptstyle\leq$}}{\scriptscriptstyle\leq}}}

   \def\DHLhksqrt#1#2{%
   \setbox0=\hbox{$#1\sqrt{#2\,}$}\dimen0=\ht0
   \advance\dimen0-0.2\ht0
   \setbox2=\hbox{\vrule height\ht0 depth -\dimen0}%
   {\box0\lower0.4pt\box2}}
\newcommand{\boolval}[1]{\mathopen{\lbrack\!\lbrack}\,#1\,\mathclose{\rbrack\!\rbrack}}
\def\[#1]{\boolval{#1}}
\newbox\gnBoxA
\newbox\gnBoxB
\newdimen\gnCornerHgt
\setbox\gnBoxA=\hbox{\tiny$\ulcorner$}
\global\gnCornerHgt=\ht\gnBoxA
\newdimen\gnArgHgt
\def\gcode #1{%
\setbox\gnBoxA=\hbox{$#1$}%
\setbox\gnBoxB=\hbox{$\bar #1$}%
\gnArgHgt=\ht\gnBoxB%
\ifnum     \gnArgHgt<\gnCornerHgt \gnArgHgt=0pt%
\else \advance \gnArgHgt by -\gnCornerHgt%
\fi \raise\gnArgHgt\hbox{\tiny$\ulcorner$} \box\gnBoxA %
\raise\gnArgHgt\hbox{\tiny$\urcorner$}}
\newcommand{\UnderTilde}[1]{{\setbox1=\hbox{$#1$}\baselineskip=0pt\vtop{\hbox{$#1$}\hbox to\wd1{\hfil$\sim$\hfil}}}{}}
\newcommand{\Undertilde}[1]{{\setbox1=\hbox{$#1$}\baselineskip=0pt\vtop{\hbox{$#1$}\hbox to\wd1{\hfil$\scriptstyle\sim$\hfil}}}{}}
\newcommand{\undertilde}[1]{{\setbox1=\hbox{$#1$}\baselineskip=0pt\vtop{\hbox{$#1$}\hbox to\wd1{\hfil$\scriptscriptstyle\sim$\hfil}}}{}}
\newcommand{\UnderdTilde}[1]{{\setbox1=\hbox{$#1$}\baselineskip=0pt\vtop{\hbox{$#1$}\hbox to\wd1{\hfil$\approx$\hfil}}}{}}
\newcommand{\Underdtilde}[1]{{\setbox1=\hbox{$#1$}\baselineskip=0pt\vtop{\hbox{$#1$}\hbox to\wd1{\hfil\scriptsize$\approx$\hfil}}}{}}

\def\<#1>{\left\langle#1\right\rangle}

\newcommand{\QEDbox}{\fbox{}}

\newcommand{\cell}[1]{\boxit{\hbox to 17pt{\strut\hfil$#1$\hfil}}}
\newcommand{\head}[2]{\lower2pt\vbox{\hbox{\strut\footnotesize\it\hskip3pt#2}\boxit{\cell#1}}}
\newcommand{\boxit}[1]{\setbox4=\hbox{\kern2pt#1\kern2pt}\hbox{\vrule\vbox{\hrule\kern2pt\box4\kern2pt\hrule}\vrule}}
\newcommand{\Col}[3]{\hbox{\vbox{\baselineskip=0pt\parskip=0pt\cell#1\cell#2\cell#3}}}
\newcommand{\tapenames}{\raise 5pt\vbox to .7in{\hbox to .8in{\it\hfill input: \strut}\vfill\hbox to
.8in{\it\hfill scratch: \strut}\vfill\hbox to .8in{\it\hfill output: \strut}}}
\newcommand{\Head}[4]{\lower2pt\vbox{\hbox to25pt{\strut\footnotesize\it\hfill#4\hfill}\boxit{\Col#1#2#3}}}
\newcommand{\Dots}{\raise 5pt\vbox to .7in{\hbox{\ $\cdots$\strut}\vfill\hbox{\ $\cdots$\strut}\vfill\hbox{\
$\cdots$\strut}}}